\documentclass[a4paper,11pt]{article}
\usepackage{cite}
\usepackage{setspace,caption}
\usepackage{amsmath,amsthm, amssymb}
\usepackage{wasysym}
\usepackage{amsfonts}
\usepackage{graphicx}
\usepackage{subfig,tikz}
\usepackage{url}
\usepackage{bm}
\usepackage{algorithmic,algorithm}
\allowdisplaybreaks[1]

\DeclareMathOperator*{\intr}{int}

\newcommand{\Tcal}{\mathcal{T}}
\newcommand{\Pcal}{\mathcal{P}}

\newcommand{\Ccal}{\mathcal{C}}

\newcommand{\onebld}{{\mbox{\bf 1}}}
\newcommand{\wt}{\widetilde}

\newcommand{\ol}{\overline}

\newtheorem{theorem}{Theorem}
\newtheorem{lemma}{Lemma}
\newtheorem{corollary}{Corollary}
\newtheorem{proposition}{Proposition}
\newtheorem{definition}{Definition}

\newcommand{\bb}{\mathbb}

\newcommand{\R}{\bb R}

\newcommand{\Z}{\bb Z}

\begin{document}
\title{Approximation of Minimal Functions by Extreme Functions}
\author{Teresa M. Lebair and Amitabh Basu}
\date{\today}
\maketitle
%
%
\begin{abstract}
In a recent paper, Basu, Hildebrand, and Molinaro established that the set of continuous minimal functions for the 1-dimensional Gomory-Johnson infinite group relaxation possesses a dense subset of extreme functions.    
The $n$-dimensional version of this result was left as an open question. In the present paper, we settle this query in the affirmative:~for any integer $n \geq 1$, every continuous minimal function can be arbitrarily well approximated by an extreme function in the $n$-dimensional Gomory-Johnson model. 
\end{abstract}
%
%
\section{Introduction}\label{sec:intro}
%
%

Gomory and Johnson discovered that the study of cutting plane theory for integer programming could greatly benefit from the analysis of the class of minimal functions.
They published their results in a series of groundbreaking papers (c.f.~\cite{infinite,infinite2}), written in the early 1970's.  
Interest in minimal functions then resurfaced in the early 2000's, with intense activity in the past 7--8 years; see the surveys~\cite{basu2016light,basu2016light2,basu2015geometric} and the references therein.  We define such functions as follows.
\begin{definition}\label{def:strong_min}
Let $n$ be a natural number and $b \in \mathbb R^n \setminus \mathbb Z^n$.  
We say that $\pi : \mathbb R^n \rightarrow \mathbb R$ is a {\bf minimal function} if the following three conditions are satisfied:
\begin{itemize}
\item[(C1)] $\pi(z) = 0$ for all $z \in \mathbb Z^n$, and $\pi \geq 0$,
\item[(C2)] $\pi$ satisfies the {\em symmetry property}, i.e., for all $x \in \mathbb R^n$,
$
\pi(x)+\pi(b-x) = 1,
$ and
\item[(C3)] $\pi$ is {\em subadditive}, i.e., $\pi(x+y) \leq \pi(x)+\pi(y)$ for all $x, y \in \mathbb R^n$.
\end{itemize}
\end{definition}
%
%
It is not hard to see that a minimal function $\pi$ is {\bf periodic with respect to $\mathbb Z^n$}, i.e., $\pi(x+z) = \pi(x) \text{ for all } x \in \mathbb R^n, z \in \mathbb Z^n$ (see, for example,~\cite{basu2016light,basu2016light2,basu2015geometric}).
%
%

Minimal functions are closely linked with integer programming.  
Consider the feasible region of a pure integer program written in the simplex tableaux form
$$\Big\{(x, y) \in \Z^n_+ \times \Z^k_+ \,\Big|\, x + b = \sum_{i=1}^k  p^{(i)} y_i\Big\},$$ 
where $b, p^{(1)}, \dots, p^{(k)} \in \R^n$. We observe that for any feasible $(x,y)$ and any minimal function $\pi: \R^n \to \R$, we have 
%
$$
\sum_{i=1}^k  \pi(p^{(i)}) \, y_i \geq \pi\Big(\sum_{i=1}^k p^{(i)} y_i \Big) = \pi(x + b) = \pi(b) = 1,
$$
%
where the first inequality follows from the subadditivity of $\pi$ (condition (C3) above), the second equality follows from the periodicity of $\pi$, and the last equality follows from conditions (C1) and (C2). Thus, the inequality $\sum_{i=1}^k \pi(p^{(i)}) \, y_i \geq 1$ is a {\em valid inequality} for the feasible region of our integer program. 
Note that the simplex tableaux solution $(x,y) = (-b, 0)$ does not satisfy this inequality. Therefore, adding such inequalities to the initial linear programming relaxation yields a tighter system.
This observation illustrates the importance of minimal functions in the integer programming literature. 

Observe that the family of minimal functions is a convex set in the (infinite dimensional) space of functions on $\R^n$; convex combinations of any two minimal functions satisfy conditions (C1)-(C3).  
In fact, under the natural topology of pointwise convergence, the family of minimal functions forms a compact convex set. This motivates the study of the extreme points of the set of minimal functions, i.e., minimal functions $\pi$ that are not the convex combination of two other distinct minimal functions. Such functions are naturally called {\em extreme functions}. A significant part of the recent effort in understanding minimal functions has been spent on characterizing and understanding the properties of these extreme functions.

We will restrict our attention to {\em continuous} minimal functions in this paper. It was recently shown in~\cite{basu2016structure} that (i) it is sufficient to consider only continuous minimal functions in regards to the application of integer programming with rational data and (ii) that focusing on continuous minimal functions is not very restrictive, in a precise mathematical sense. 
Note that the set of continuous minimal functions is also a convex subset of the set of minimal functions.

In~\cite{basu2016minimal}, Basu, Hildebrand, and Molinaro observed a curious property for the convex set of continuous minimal functions $\pi:\R\to \R$, i.e., with $n=1$. Moreover, they demonstrated that the extreme points of this convex set are dense in this set (under the topology of pointwise convergence). In fact, they established a slightly stronger result: for any continuous minimal function $\pi: \R \to \R$ and $\epsilon > 0$, there exists a continuous extreme function $\pi^*: \R \to \R$ such that $|\pi(x) - \pi^*(x)| \leq \epsilon$ for all $x \in \R$, i.e., $\pi^*$ is shown to be extreme in the set of all minimal functions, not only in the set of continuous minimal functions. This observation is clearly an infinite-dimensional phenomenon; compact convex sets in finite dimensions do not possess dense subsets of extreme points. The question as to whether the ``density" result of~\cite{basu2016minimal} continues to hold for general integers $n \geq 2$ was unresolved. 
In fact, for a related family of functions -- the so-called {\em cut generating functions for the continuous model} -- an analogous result was established to be true only for $n=1, 2$, and false for $n \geq 3$. This created even more mystery for the Gomory-Johnson model, with $n \geq 2$, that we consider in the present paper.

{\bf This paper resolves the above ``density" question and shows that for all $n\geq 1$, continuous extreme functions form a dense subset of the set of continuous minimal functions.} While the proof strategy in this paper borrows ideas from the 1-dimensional proof in~\cite{basu2016minimal}, there are significant technical barriers that must be surmounted in order to obtain the general $n$-dimensional result.  For example, the fact that any continuous minimal function can be arbitrarily well-approximated by a {\em piecewise linear} minimal function can be derived easily via interpolation for $n=1$, and is the first step towards the 1-dimensional proof in~\cite{basu2016minimal}. The absence of such a result for $n\geq 2$ is a major obstacle in extending the density result to higher dimensions. We show that this piecewise linear approximation does indeed hold for general $n\geq 2$, by a series of careful and delicate constructions (c.f.~Corollary~\ref{cor:PWL-approx}). In our opinion, this result alone is likely to be a very useful tool for subsequent research in the area. Similar non-trivial departures from the techniques of~\cite{basu2016minimal} are required to obtain our main result.  


This paper is organized as follows.  In Section~\ref{sec:main_result}, we formally state the main result of this paper and give an overview of the proof. 
Additionally, we lay out the notation and definitions used in the rest of the paper in this section.  In Sections~\ref{sec:pwl_approx},~\ref{sec:perturb}, and~\ref{sec:sym_fill_in} we construct several approximations of a given continuous, minimal $\pi: \mathbb R^n \rightarrow \mathbb R$, the last one of which is an extreme function of the set of  minimal functions.  Finally, in Section~\ref{sec:proof}, we assemble the complete proof of the main result. 
%
%

%
%
\section{Main Result and Overview of Proof}\label{sec:main_result}
%
%
The main result of this paper is formally given by the following Theorem.
%
%
\begin{theorem}\label{thm:main_thm}
Let $n \geq 1$ be any natural number and $b \in \mathbb Q^n \setminus \mathbb Z^n$. Let $\pi: \mathbb R^n \rightarrow \mathbb R$ be any continuous, minimal function, and let $\epsilon > 0$ be given.  Then there exists an extreme function $\pi^\prime: \mathbb R^n \rightarrow \mathbb R$ such that 
$\| \pi - \pi^\prime\|_\infty < \epsilon$. 
\end{theorem} 
%
%
Note that by the periodicity of the set of minimal functions, we may assume that $b \in ([0,1)^n \cap \mathbb Q^n)\setminus\{0\}$ without loss of generality, as we will do in the rest of this paper. 

The remainder of this work is devoted to proving Theorem~\ref{thm:main_thm}.  
Inspired by~\cite{basu2016minimal}, the steps utilized to build $\pi^\prime$ are outlined below. 
 Moreover, the following list of steps describes the development of a sequence of functions (i.e., $\wt\pi$, $\pi_{comb}$, $\pi_{fill-in}$, and $\pi_{sym}$) that each approximate $\pi$ and have certain desirable properties.  The last function in this sequence $\pi_{sym}$ is the desired extreme function $\pi^\prime$:
\begin{itemize}
\item[(i)] We first approximate the continuous,  minimal function $\pi$ by a piecewise linear function $\wt\pi$, which has a number of useful attributes, e.g., Lipschitz continuity (c.f.~Propsition~\ref{prop:pwl}), but is not necessarily subadditive.  
\item[(ii)] We next construct a piecewise linear minimal function $\pi_{comb}$, by perturbing $\wt\pi$ to regain subadditivity.  Furthermore, taking a suitable perturbation of $\wt \pi$ will ensure that
\begin{equation}\label{eqn:Delta_pi_lift_positive}
\pi_{comb}(x) + \pi_{comb}(y) - \pi_{comb}(x+y) > \frac{\epsilon}{4} > 0 \text{ for ``most" } x,y \in \mathbb R^n 
\end{equation}
(c.f.~(iii) of Proposition~\ref{prop:pi_comb}).  This will allow us to make slight modifications to $\pi_{comb}$ in the subsequent constructions, while maintaining subadditivity.    
\item[(iii)] Just as a symmetric 2-slope fill-in procedure is used in~\cite{basu2016minimal}, we use a symmetric $(n+1)$-slope fill-in procedure to produce the extreme function $\pi_{sym}$.  We first apply the (asymmetric) $(n+1)$-slope fill-in procedure of~(\ref{eqn:fill_in}) to $\pi_{comb}$ in order to obtain $\pi_{fill-in}$ (c.f.~(\ref{eqn:pi_fill_in})) and then symmetrize $\pi_{fill-in}$ via~(\ref{eqn:pi_sym}) to obtain $\pi_{sym}$.  We will see that $\pi_{sym}$ is a piecewise linear $(n+1)$-slope (c.f.~Definition~\ref{def:slope}), genuinely $n$-dimensional (c.f.~Definition~\ref{def:genuine}),  minimal function (c.f.~Proposition~\ref{prop:pi_sym}).  Hence, we may apply~\cite[Theorem 1.7]{basu-hildebrand-koeppe-molinaro:k+1-slope} to establish that $\pi^\prime \equiv \pi_{sym}$ is extreme.  We will also see that $\| \pi - \pi^\prime\|_\infty < \epsilon$ (c.f.~Section~\ref{sec:proof}). The subadditivity of $\pi^\prime \equiv \pi_{sym}$ roughly follows from~(\ref{eqn:Delta_pi_lift_positive}).
\end{itemize}
%
%


\paragraph{Notation and terminology.} For $n \in \mathbb N$, $[n]$ denotes the set $\{1,\dots,n\}$, and $e^{(1)}, \dots, e^{(n)} \in \mathbb R^n$ denote the 1-st through $n$-th basic unit vectors.  
Let $\onebld \in \mathbb R^n$ denote the vector of all ones.  
We also define the set of permutations 
\[\Sigma_n :=\left\{\sigma = (i_1,\dots,i_n) \,|\, \sigma \text{ is a permutation of } (1,\dots,n) \right\}.\]
Additionally, $v_i$ denotes the $i$-th component of the vector $v$.  
For a set $A \subseteq \mathbb R^n$, let $\text{conv}(A)$, $\text{cone}(A)$, $\text{int}(A)$, and $\text{relint}(A)$ denote the convex hull, the positive hull, the interior, and the relative interior of the set $A$, respectively.  Let $\bar B_\delta^\infty(x)$ denote the closed ball centered at $x \in \mathbb R^n$ with radius $\delta > 0$, with respect to the $\infty$-norm. 
For a function $f:\mathbb R^n \rightarrow \mathbb R$ and $x,d \in \mathbb R^n$, let $f^\prime(x;d)$ denote the directional derivative of $f$ at $x$ in the direction $d$ and $\nabla f(x)$ denote the gradient of $f$ at $x$, whenever these objects exist.  
Finally, for any $\pi:\mathbb R^n \rightarrow \mathbb R$, we define $\Delta \pi: \mathbb R^n \times \mathbb R^n \rightarrow \mathbb R$ such that
\[
\Delta \pi(x,y) := \pi(x) +\pi(y)-\pi(x+y).
\]
Consequently, $\pi$ is subadditive if and only if $\Delta\pi \geq 0$ everywhere. This fact will be used many times throughout this paper, without further comment.

Our goal is to construct an extreme function $\pi^\prime$ that is piecewise linear.  Therefore, we begin with a short review of piecewise linear functions. In order to discuss piecewise linear functions, we first introduce the notion of a polyhedral complex.
%
%
\begin{definition}\cite[Definition~5.1]{z}\label{def:poly_com} A {\bf polyhedral complex} $\Ccal$ is a finite collection of cells (i.e., polyhedra) in $\mathbb R^n$ (including the empty polyhedron) such that 
\begin{itemize}
\item[(i)] if $P \in \Ccal$, then each face of $P$ is also in $\Ccal$, and 
\item[(ii)] for any $P, Q \in \Ccal$, $P \cap Q$ is a face of both $P$ and $Q$.
\end{itemize}
\end{definition}
A cell $P \in \Ccal$ is {\bf maximal} if there does not exist any $Q \in \Ccal$ properly containing $P$.  Additionally, a complex $\Ccal$ is {\bf pure} if all maximal cells in $\Ccal$ have the same dimension, and {\bf complete} if $\cup_{P \in \Ccal} P = \mathbb R^n$.  
Finally, we say that a polyhedral complex $\Ccal$ is a {\bf triangulation} if each $P \in \Ccal$ is a simplex.

For fixed $\delta > 0$, we define $U_\delta$ to be the group generated by the vectors $\delta e^{(1)},\dots,\delta e^{(n)}$ (under the usual addition operation).  Then let $\Tcal_\delta$ denote the triangulation containing the elements and all faces of the elements (including the empty face) in the set
\[
 \Big\{ \text{conv}\Big(\Big\{u, u+\delta \sum_{j=1}^1 e^{(i_j)}, \dots, u+\delta \sum_{j=1}^n e^{(i_j)}\Big\}\Big) \; \Big| \; u \in U_\delta,\, (i_1,\dots,i_n) \in \Sigma_n  \Big\}.
\]
Note that $\Tcal_\delta$ is the well-known polyhedral complex called the {\em Coxeter-Freudenthal-Kuhn} triangulation (c.f.~\cite[pp.~136]{allgower1997numerical}).  

\begin{definition}\label{def:pwl} Let $\Ccal$ be a pure, complete polyhedral complex in $\mathbb R^n$.  Consider a function $\theta : \mathbb R^n \rightarrow \mathbb R$, where, for each $P \in \Ccal$, there exists a vector $a^{P} \in \mathbb R^n$ and a constant $\delta_P$ satisfying $\theta(x) = \langle a^{P}, x \rangle + \delta_P$ for all $x \in P$.  Then $\theta$ is called a {\bf piecewise linear function}, more specifically, a {\bf piecewise linear function with cell complex $\Ccal$}.
\end{definition}

Observe that Definition~\ref{def:pwl} implies that any piecewise linear function is automatically continuous.

%
%

\section{Piecewise Linear Approximations}\label{sec:pwl_approx}

The main goal of this section is to establish Proposition~\ref{prop:pwl}, in which we show the existence of $\wt \pi$, a piecewise linear approximation of a given continuous minimal $\pi$, that holds several favorable properties.  We will see that Proposition~\ref{prop:pwl} proceeds immediately from Lemmas~\ref{lem:pwl} and~\ref{lem:pi_adjust}.
%
%
\begin{lemma}\label{lem:pwl}
Suppose that $b \in ([0,1)^n \cap \mathbb Q^n)\setminus\{0\}$ and $\pi:\mathbb R^n \rightarrow \mathbb R$ is a continuous,  minimal function.  Then for each $\epsilon >0$, there exists $\pi_{pwl}:\mathbb R^n \rightarrow \mathbb R$ and $\delta_1 > 0$ such that $b \in U_{\delta_1}$ and the following items hold:
(i) $\pi_{pwl}$ is piecewise linear with cell complex $\Tcal_{\delta_1}$; 
(ii) $\| \pi - \pi_{pwl} \|_\infty < \epsilon$; 
(iii) $(\pi_{pwl} + 3\epsilon)$ is subadditive; 
(iv) $\pi_{pwl}$ satisfies conditions (C1) and (C2) in Definition~\ref{def:strong_min}; 
(v) $\pi_{pwl}$ is periodic with respect to $\mathbb Z^n$; and 
(vi) $\pi_{pwl}$ is Lipschitz continuous.
\end{lemma}
%
%
\begin{proof}
Since $\pi$ is continuous on $[-1,1]^n$, a compact set, $\pi$ must also be uniformly continuous on $[-1,1]^n$.  Therefore, for $\epsilon >0$, there exists $\delta_1 >0$ such that $\|x-y\|_\infty \leq \delta_1$ implies that $|\pi(x)-\pi(y)| < \epsilon$ for all $x, y \in [-1,1]^n$, and consequently, for all $x,y \in \mathbb R^n$ due to the periodicity of $\pi$.  
Without loss of generality, we can select $\delta_1$ such that $\delta_1 = p^{-1}$ for some $p \in \mathbb N$ and $b \in U_{\delta_1}$.    
Then for any maximal $T \in \Tcal_{\delta_1}$ with vertices $u^{(1)}, \dots,u^{(n+1)}$, we have
\begin{equation}\label{eqn:dist2vertex}
\|u^{(i)} - x\|_\infty \leq  \delta_1, \text{ so } |\pi(u^{(i)}) - \pi(x)| < \epsilon \text{ for all } x \in T \text{ and } i \in [n+1].
\end{equation}  
We make the following observations.

(i)  In view of~\cite[pp.~15]{hudson1969piecewise}, we let $\pi_{pwl}$ be the unique piecewise linear function with cell complex $\Tcal_{\delta_1}$ such that $\pi_{pwl}(u) = \pi(u)$ for each $u\in U_{\delta_1}$.  Hence, (i) is satisfied.  

(ii) Let $x \in \mathbb R^n$ so that $x \in T$ for some maximal $T \in \Tcal_{\delta_1}$.  Then $x$ can be written as a convex combination of the vertices $u^{(1)}, \dots, u^{(n+1)} \in U_{\delta_1}$ of $T$, i.e., $x = \sum_{i=1}^{n+1} \lambda_i u^{(i)}$, where each $\lambda_i \geq 0$ and $\sum_{i=1}^{n+1} \lambda_i = 1$.
Observe via~(\ref{eqn:dist2vertex}) that (ii) holds, as
\[
|\pi(x) - \pi_{pwl}(x)| = \Big|\pi(x) - \sum_{i=1}^{n+1}  \lambda_i\pi_{pwl}(u^{(i)})\Big| = \Big|\pi(x) - \sum_{i=1}^{n+1}  \lambda_i\pi(u^{(i)})\Big| \leq \sum_{i=1}^{n+1}  \lambda_i |\pi(x) - \pi(u^{(i)})| < \epsilon.
\]
  
(iii) Note that by (ii) and the subadditivity of $\pi$,
\begin{equation}\label{eqn:Delta_pi_pwl}
\Delta(\pi_{pwl} + 3\epsilon)(x,y) = \pi_{pwl}(x) + \pi_{pwl}(y) - \pi_{pwl}(x+y) + 3\epsilon \geq \Delta \pi(x,y) \geq 0.
\end{equation}

(iv) The function $\pi_{pwl}$ satisfies (C1) since $\mathbb Z^n \subseteq U_{\delta_1}$, by choice of $\delta_1$, and the fact that the linear interpolant of a nonnegative function is nonnegative.  To see that $\pi_{pwl}$ satisfies (C2), note that if $x \in T \in \Tcal_{\delta_1}$ has vertices $u^{(1)}, \dots, u^{(n+1)}$, then $b -x \in b-T \in \Tcal_{\delta_1}$, where the simplex $b-T$ has vertices $b-u^{(1)}, \dots, b-u^{(n+1)}$, as $b \in U_{\delta_1}$.  Therefore, by the piecewise linear structure of $\pi_{pwl}$, we have that there exist $\lambda_i$'s, $i \in [n+1]$, such that each $\lambda_i \geq 0$ and $\sum_{i=1}^{n+1} \lambda_i = 1$, so that $x = \sum_{i=1}^{n+1} \lambda_i u^{(i)}$, 
\[
b - x = b -  \sum_{i=1}^{n+1} \lambda_i u^{(i)} =  \sum_{i=1}^{n+1} \lambda_i (b-u^{(i)}),\quad\text{ and }
\]
\[
\pi_{pwl}(x) + \pi_{pwl}(b-x) = \sum_{i = 1}^{n+1} \lambda_i \pi_{pwl}(u^{(i)}) + \sum_{i = 1}^{n+1} \lambda_i \pi_{pwl}(b - u^{(i)}) = \sum_{i=1}^{n+1} \lambda_i (\pi(u^{(i)}) + \pi(b-u^{(i)}) ) = 1.
\]  
Hence, $\pi_{pwl}$ satisfies condition (C2).

(v) The periodicity of $\pi_{pwl}$ follows from the periodicity of $\pi$ and the piecewise linear structure of $\pi_{pwl}$ with cell complex $\Tcal_{\delta_1}$.  

(vi) Note that $[0,1]^n$ consists of the union of only finitely many maximal $T \in \Tcal_{\delta_1}$ for fixed $\delta_1 > 0$, where $\pi_{pwl}$ is affine on each such $T$.  Hence, the gradient (when it exists) of $\pi_{pwl}$ on such $T \in \Tcal_{\delta_1}$ can only be equal to one of only finitely many different vectors, which we list as $a^{(1)}, \dots, a^{(N)}$.  Thus, $\pi_{pwl}$ is Lipschitz continuous on $[0,1]^n$, with Lipschitz constant $L := \max_{i=1}^N \|a^{(i)}\|_1$ with respect to the $\infty$-norm via~\cite[Proposition 2.2.7]{scholtes2012introduction}.  The Lipschitz continuity of $\pi_{pwl}$ on $\mathbb R^n$ follows from (v).
\end{proof}
%
%
%
%
We now construct $\pi_{adjust}$, by flattening the piecewise linear function $\pi_{pwl}$ (c.f.~Lemma~\ref{lem:pwl}) in a tiny $\infty$-norm neighborhood of each point in $(b+\mathbb Z^n)\cup \mathbb Z^n$.  
In particular, in each $\infty$-norm neighborhood of $z \in \mathbb Z^n$, we will have $\pi_{adjust} = 0$, and in each $\infty$-norm neighborhood of $b+z \in b+\mathbb Z^n$, we will have $\pi_{adjust} = 1$.   
For $\epsilon > 0$, $b \in ([0,1)^n \cap \mathbb Q^n)\setminus\{0\}$, and some continuous, minimal $\pi: \mathbb R^n \rightarrow \mathbb R$, let $\pi_{pwl}$ be a function that satisfies the conclusions of Lemma~\ref{lem:pwl} for some suitable $\delta_1 > 0$ and Lipschitz constant $L$ (with respect to the $\infty$-norm).  
Suppose that $\delta_2 \in \mathbb R$ satisfies 
\begin{equation}\label{eqn:delta_12_1} 
\delta_2 < \min \left\{\frac{1}{4Ln}, \frac{\epsilon}{2L}\right\} \quad \text{and} \quad
 \delta_2 = \frac{\delta_1}{q} \text{ for some } q \in \mathbb N,
\end{equation}
so that $\mathbb Z^n \subseteq U_{\delta_1} \subseteq U_{\delta_2}$ and $b \in U_{\delta_1} \subseteq U_{\delta_2}$.  Furthermore assume that $\delta_2 > 0$ is small enough so that
\begin{equation}\label{eqn:delta_12_2} 
\bar B_{2\delta_2}^\infty(z^{(1)}) \cap \bar B_{2\delta_2}^\infty(b-z^{(2)})   = \emptyset \text{ for all } z^{(1)}, z^{(2)} \in \mathbb Z^n.
\end{equation}
%
Define $\pi_{adjust}:\mathbb R^n \rightarrow \mathbb R$, a piecewise linear function with cell complex $\Tcal_{\delta_2}$, where, for each $u \in U_{\delta_2}$, 
\begin{equation}\label{eqn:pi_adjust}
\pi_{adjust}(u) =
\begin{cases}
0			 	& \text{if }  u \in \bar B_{\delta_2}^\infty(z) \text{ for some }  z \in \mathbb Z^n,\\
1				& \text{if }  u \in \bar B_{\delta_2}^\infty(b-z) \text{ for some }  z \in \mathbb Z^n,\\
\pi_{pwl}(u)	& \text{otherwise},
\end{cases}
\end{equation}
and the values of $\pi_{adjust}$ on $\R^n\setminus U_{\delta_2}$ are defined via linear interpolation. We note that $\pi_{adjust}$ is well-defined via~(\ref{eqn:delta_12_2}).  The next lemma describes several desirable properties of $\pi_{adjust}$.
%
%
\begin{lemma}\label{lem:pi_adjust}  
Let $b \in ([0,1)^n \cap \mathbb Q^n)\setminus\{0\}$, $\epsilon > 0$, and fix a continuous,  minimal function $\pi:\mathbb R^n \rightarrow \mathbb R$.  For suitable $\delta_1 > 0$ (c.f.~Lemma~\ref{lem:pwl}), let $\pi_{pwl}:\mathbb R^n \rightarrow \mathbb R$ be any function with Lipschitz constant $L$ (with respect to the $\infty$-norm) that satisfies the conclusions of Lemma~\ref{lem:pwl}.  
Consider the function $\pi_{adjust}$ given by~(\ref{eqn:pi_adjust}) for some appropriately chosen $\delta_2 > 0$ (c.f.~(\ref{eqn:delta_12_1}),~(\ref{eqn:delta_12_2})).  
The following items hold:   
(i)  $\pi_{adjust}$ satisfies conditions (C1) and (C2); 
(ii) $\pi_{adjust}$ is Lipschitz continuous with Lipschitz constant $4Ln$ (with respect to the $\infty$-norm); 
(iii) $\|\pi_{pwl} - \pi_{adjust} \|_\infty < \epsilon$;  
(iv) $(\pi_{adjust} + 6\epsilon)$ is subadditive; and
(v) $\pi_{adjust}$ is periodic with respect to $\mathbb Z^n$.
\end{lemma}
%
%
\begin{proof}

(i)  It is not hard to see that $\pi_{adjust}$ satisfies (C1), as $\pi_{adjust}(u) \geq 0$ for all $u \in U_{\delta_2}$, the linear interpolant of any nonnegative function is nonnegative, and $\pi_{adjust}(z) = 0$ for all $z \in \mathbb Z^n$.  To see that $\pi_{adjust}$ satisfies (C2), note that $\pi_{adjust}(u) + \pi_{adjust}(b-u) = 1$ for all $u \in U_{\delta_2}$, as $\pi_{pwl}$ satisfies (C2), and then apply an argument similar to that used in the proof of (iv) in Lemma~\ref{lem:pwl}.

(ii)  We next show that $\pi_{adjust}$ is Lipschitz continuous.  Fix $u \in U_{\delta_2}$ and $(i_1,\dots,i_n) \in \Sigma_n$, so 
\[
T := \text{ conv}\Big(\Big\{u,u+\delta_2 e^{(i_1)},\dots,u+\delta_2 \sum_{j=1}^n e^{(i_j)} \Big\}\Big) 
\]
is an arbitrary maximal cell of $\Tcal_{\delta_2}$.  
Since $\pi_{adjust}$ is affine on $T$, there exists $a_T \in \mathbb R^n$ and $\delta_T \in \mathbb R$ such that $\pi_{adjust}(x) = \langle  a_T, x \rangle + \delta_T$ for all $x \in T$.  In particular, $a_T$ is the unique vector satisfying
\[
\pi_{adjust}\Big(u + \delta_2 \sum_{j=1}^k e^{(i_j)} \Big) - \pi_{adjust}(u) = \Big\langle a_T , \delta_2 \sum_{j=1}^k e^{(i_j)} \Big\rangle \quad \text{for all} \quad k \in [n].
\]
Defining $u^{(k)} := u + \delta_2 \sum_{j=1}^k e^{(i_j)}$ for each $k \in [n]$, we may rewrite the above display as
\[
\delta_2
\begin{bmatrix}
1 &    && \\
1 & 1 & & \\
\vdots & \vdots &\ddots& \\
1 & 1  & \dots&1
\end{bmatrix}
 \left[\begin{array}{c c c  c c}
&& e^{(i_1)} &&\\
\hline
& & e^{(i_2)} &&\\
\hline
& &\vdots &\\
\hline
& & e^{(i_n)}&&
\end{array}\right]
 a_T =
\begin{bmatrix}
\pi_{adjust}(u^{(1)}) - \pi_{adjust}(u)\\
\pi_{adjust}(u^{(2)}) - \pi_{adjust}(u)\\
\vdots\\
\pi_{adjust}(u^{(n)}) - \pi_{adjust}(u)
\end{bmatrix}.
\]
Hence,
\[
a_T = \delta^{-1}_2  
\left[\begin{array}{c| c| c|  c}
&&&\\
&&&\\
e^{(i_1)} & e^{(i_2)} & \dots & e^{(i_n)}\\
&&&\\
&&&
\end{array}\right]
\begin{bmatrix}
1 &    && \\
-1 & 1 & & \\
 & \ddots &\ddots& \\
 &   &-1 &1
\end{bmatrix}
\begin{bmatrix}
\pi_{adjust}(u^{(1)}) - \pi_{adjust}(u)\\
\pi_{adjust}(u^{(2)}) - \pi_{adjust}(u)\\
\vdots\\
\pi_{adjust}(u^{(n)}) - \pi_{adjust}(u)
\end{bmatrix},
\]
which implies
\begin{equation}\label{eqn:a_T_bound}
\| a_T \|_1 \leq 2\, n\,  \delta^{-1}_2  \max_{i=1}^n |\pi_{adjust}(u^{(i)}) - \pi_{adjust}(u)|.
\end{equation}
Now, if $u, u^{(i)} \notin \bar B_{\delta_2}^\infty(z) \cup \bar B_{\delta_2}^\infty(b-z)$ for all $z \in \mathbb Z^n$, then 
\begin{equation}\label{eqn:pi_adj_Lip1}
|\pi_{adjust}(u^{(i)}) - \pi_{adjust}(u)| = |\pi_{pwl}(u^{(i)}) - \pi_{pwl}(u)| \leq L \delta_2,
\end{equation}
by (vi) of Lemma~\ref{lem:pwl}.
Otherwise, at least one of $u$, $u^{(i)}$ belongs to $ \bar B_{\delta_2}^\infty(z)$ or $ \bar B_{\delta_2}^\infty(b-z)$ for some $z \in \mathbb Z^n$.  

Suppose that at least one of $u$ or $u^{(i)} \in \bar B_{\delta_2}^\infty(z)$ so that $\pi_{adjust}(u) =0 $ or $\pi_{adjust}(u^{(i)})=0 $.  In this case, we must have $u, u^{(i)} \in \bar B_{2\delta_2}^\infty(z)$ as $u$ and $u^{(i)}$ are vertices of the same simplex $T \in \Tcal_{\delta_2}$.  This leads to
\begin{equation}\label{eqn:pi_adj_Lip2}
|\pi_{adjust}(u^{(i)}) - \pi_{adjust}(u)|  \leq  \max_{w \in U_{\delta_2} \cap \bar B_{2\delta_2}^\infty (z) }\pi_{adjust}(w) \leq \max_{w \in U_{\delta_2} \cap \bar B_{2\delta_2}^\infty (z) }\pi_{pwl}(w) \leq 2L\delta_2, 
\end{equation}
by~(\ref{eqn:delta_12_2}),~(\ref{eqn:pi_adjust}), and conclusions (iv) and (vi) of Lemma~\ref{lem:pwl}.

Alternatively, suppose that at least one of $u$ or $u^{(i)} \in \bar B_{\delta_2}^\infty(b-z)$ so that $\pi_{adjust}(u) = 1$ or $\pi_{adjust}(u^{(i)}) = 1$.  In this case, we must have $u, u^{(i)} \in \bar B_{2\delta_2}^\infty(b-z)$ as, again, $u$ and $u^{(i)}$ are vertices of the same simplex $T \in \Tcal_{\delta_2}$.  We then obtain 
\begin{align}
|\pi_{adjust}(u^{(i)}) - \pi_{adjust}(u)|  &\leq  \max_{w \in U_{\delta_2} \cap \bar B_{2\delta_2}^\infty (b-z) }  | 1 - \pi_{pwl}(w)| \label{eqn:pi_adj_Lip3} \\
&=   \max_{w \in U_{\delta_2} \cap \bar B_{2\delta_2}^\infty (b-z)} | \pi_{pwl}(b-z) - \pi_{pwl}(w)| \leq 2L\delta_2,\notag
\end{align}
via~(\ref{eqn:delta_12_2}),~(\ref{eqn:pi_adjust}), and conclusions (iv) and (vi) of Lemma~\ref{lem:pwl}.

Combining equations~(\ref{eqn:a_T_bound})-(\ref{eqn:pi_adj_Lip3}), we have that $\|a_T\|_1 \leq 4Ln$ for all $T \in \Tcal_{\delta_2}$.  Hence $\pi_{adjust}$ is Lipschitz, with Lipschitz constant $4Ln$, with respect to the $\infty$-norm (c.f.~\cite[Proposition 2.2.7]{scholtes2012introduction}).  

(iii)
%
If $x \notin \bigcup_{z \in \mathbb Z^n} \bar B^\infty_{2\delta_2}(z)$ and $x \notin \bigcup_{z \in \mathbb Z^n} \bar B^\infty_{2\delta_2}(b-z)$, then $x \in T$ for some maximal $T \in \Tcal_{\delta_2}$ whose vertices $u^{(1)}, \dots, u^{(n+1)}$ satisfy $\| u^{(i)} - z \|_\infty > \delta_2$ and $\| u^{(i)} - b+ z \|_\infty > \delta_2$ for all $i \in [n+1]$ and $z \in \mathbb Z^n$. 
Hence, \[\pi_{adjust}(u^{(i)}) - \pi_{pwl}(u^{(i)}) = 0\] for all such $u^{(i)}$, so that
$|\pi_{adjust}(x) - \pi_{pwl}(x)| = 0$, as  $(\pi_{adjust} - \pi_{pwl})$ is affine on $T$.  

Consider then $x \in  \bar B^\infty_{2 \delta_2}(z)$ for some $z \in \mathbb Z^n$.  By the Lipschitz continuity of $\pi_{pwl}$, the fact that $0 \leq \pi_{adjust} \leq \pi_{pwl}$ on each $\bar B^\infty_{2\delta_2}(z)$ (via equation~(\ref{eqn:pi_adjust})), and~(\ref{eqn:delta_12_1}), we have that
\[
|\pi_{pwl}(x) - \pi_{adjust}(x)| \leq \pi_{pwl}(x)  \leq 2L\delta_2 < \epsilon.
\]
A similar argument demonstrates that $|\pi_{pwl}(x) - \pi_{adjust}(x)| < \epsilon$ for all $x \in \bigcup_{z \in \mathbb Z^n} \bar B^\infty_{2\delta_2}(b-z)$.  Thus, $|\pi_{pwl}(x) - \pi_{adjust}(x) | < \epsilon$ for all $x \in \mathbb R^n$, and (iii) is satisfied.  

(iv) Conclusion (iv) follows from conclusion (iii), conclusion (iii) of Lemma~\ref{lem:pwl}, and an argument similar to that used in equation~(\ref{eqn:Delta_pi_pwl}).  

(v)  The periodicity of $\pi_{adjust}$ follows from~(\ref{eqn:pi_adjust}), the piecewise linear structure of $\pi_{adjust}$, and the periodicity of $\pi_{pwl}$.
This completes the proof.
\end{proof}
%
%
We conclude this section by combining Lemmas~\ref{lem:pwl} and~\ref{lem:pi_adjust} in the following proposition. 
%
%
\begin{proposition}\label{prop:pwl}
Suppose that $b \in ([0,1)^n \cap \mathbb Q^n)\setminus\{0\}$ and $\pi:\mathbb R^n \rightarrow \mathbb R$ is a continuous,  minimal function.  Then for each $\epsilon >0$, there exists $\wt \pi:\mathbb R^n \rightarrow \mathbb R$ and $\wt\delta > 0$ sufficiently small such that each of the following statements are satisfied:
(i)     $\wt \pi$ is a piecewise linear function with cell complex $\Tcal_{\wt\delta}$; 
(ii)    $\wt \pi$ is Lipschitz continuous; 
(iii)   $\| \pi - \wt \pi \|_\infty < \frac{\epsilon}{18}$; 
(iv)   $(\wt \pi + \frac{\epsilon}{6})$ is subadditive; 
(v)    $\wt \pi$ satisfies conditions (C1) and (C2);
(vi)   $\wt \pi$ is periodic with respect to $\mathbb Z^n$; 
(vii)  $\wt \pi(x) = 0$ for all $x \in \bigcup_{z \in \mathbb Z^n} \bar B_{\wt\delta}^\infty(z)$; 
(viii) $\wt \pi(x) = 1$ for all $x \in \bigcup_{z \in \mathbb Z^n} \bar B_{\wt\delta}^\infty(b-z)$;
(ix)   $\bar B_{\tilde\delta}^\infty(z^{(1)}) \cap \bar B_{\tilde\delta}^\infty(b-z^{(2)})   = \emptyset \text{ for all } z^{(1)}, z^{(2)} \in \mathbb Z^n$.
\end{proposition}
%
%
\begin{proof}
Fix $\epsilon > 0$.  For $\frac{\epsilon}{36} >0$, there exists $\delta_1 > 0$ (with $b \in U_{\delta_1}$, $\mathbb Z^n \subseteq U_{\delta_1}$) and  piecewise linear $\pi_{pwl}:\mathbb R^n \rightarrow \mathbb R$ with cell complex $\Tcal_{\delta_1}$ that satisfies the conclusions of Lemma~\ref{lem:pwl}, with $\epsilon$ replaced by $\frac{\epsilon}{36}$.  Choose $\delta_2$ sufficiently small to fulfill~(\ref{eqn:delta_12_1}) and~(\ref{eqn:delta_12_2}), with $L$ denoting the Lipschitz constant of $\pi_{pwl}$ with respect to the $\infty$-norm and $\epsilon$ replaced by $\frac{\epsilon}{36}$.  Consider $\pi_{adjust}:\mathbb R^n \rightarrow \mathbb R$ given by~(\ref{eqn:pi_adjust}).  
Applying Lemma~\ref{lem:pi_adjust} and the triangle inequality yields that $\wt \pi \equiv \pi_{adjust}$ satisfies conclusions (i)-(vi) with $\wt\delta = \delta_2$; furthermore, $\wt\pi$ satisfies (vii) and (viii) by~(\ref{eqn:pi_adjust}). (ix) follows from~\eqref{eqn:delta_12_2}.
\end{proof}

%
%
\section{A Subadditivity Inducing Perturbation}\label{sec:perturb}
%
%
The main result of this section is Proposition~\ref{prop:pi_comb}, which establishes several properties of a piecewise linear minimal function $\pi_{comb}$; the function $\pi_{comb}$ approximates the continuous minimal function $\pi$ and enjoys an additional strict subadditivity attribute.  

For this purpose, we consider $\wt \pi:\mathbb R^n \rightarrow \mathbb R$, any piecewise linear approximation of $\pi$ that meets the conclusions of Proposition~\ref{prop:pwl} for suitable $\wt \delta > 0$.  
We then perturb $\wt \pi$ in order to form $\pi_{comb}$ (c.f.~(\ref{eqn:pi_comb})), another approximation of $\pi$.  The perturbation is carried out in a way such that not only is $\pi_{comb}$ subadditive, but $\Delta \pi_{comb}(x,y) > 0$ for ``most" pairs $(x,y) \in \mathbb R^n \times \mathbb R^n$.  
Subsequent modifications of $\pi_{comb}$ may decrease the value of  $\Delta \pi_{comb}(x,y)$ for certain $(x,y)$.  To maintain subadditivity, it is therefore advantageous for $\Delta \pi_{comb}(x,y)$ to be strictly positive whenever possible.

The perturbation is also performed such that the gradient (when it exists) of $\pi_{comb}$ near each point $z \in \mathbb Z^n$ takes only one of $(n+1)$ possible different values in $\mathbb R^n$.  This feature of $\pi_{comb}$ will prove critical when we perform the $(n+1)$-slope fill-in procedure (c.f.~(\ref{eqn:fill_in})) and establish related results in Section~\ref{sec:sym_fill_in}. 
%
%
\paragraph{An important gauge function.}
%
%
In what follows, we will perturb $\wt \pi$ in the direction of $\pi_{\delta_3}$ (c.f.~(\ref{eqn:pi_delta})).  
One important ingredient in the construction of $\pi_{\delta_3}$ is the gauge function $\gamma_{\delta_3}$ (c.f.~(\ref{eqn:gamma_delta})) of a carefully constructed simplex.
This gauge function will also be employed in the $(n+1)$-slope fill-in procedure of Section~\ref{sec:sym_fill_in}. 

Let $b \in ([0,1)^n \cap \mathbb Q^n)\setminus\{0\}$.  Fix $\delta_3 > 0$ and define the vectors
\begin{equation}\label{eqn:vi_s}
v^{(i)}_{\delta_3} :=  \delta_3(b-\onebld+ne^{(i)}) \in \mathbb R^n \; \text{ for } \; i \in [n]   \quad \text{and} \quad    v_{\delta_3}^{(n+1)} := \delta_3(b- \onebld) \in \mathbb R^n,
\end{equation}
as well as the simplex
\begin{equation}\label{eqn:S_simplices} 
 \Lambda_{\delta_3,z} \; := z + \text{conv}(\{ v_{\delta_3}^{(i)}\}_{i=1}^{n+1} ).
\end{equation}
Note that $\big(\frac{\sum_{i=1}^n b_i}{n}\big)v_{\delta_3}^{(n+1)} + \sum_{i=1}^n (\frac{1-b_i}{n})v^{(i)}_{\delta_3} = 0$, each $b_i <1$ since $b \in [0,1)^n$, and $\sum_i b_i > 0$ since $b \neq 0$. Thus, we have that $z \in \intr(\Lambda_{\delta_3,z})$, and the Minkowski gauge function $\gamma_{\delta_3}: \mathbb R^n \rightarrow \mathbb R$ given by 
\begin{equation}\label{eqn:gamma_delta}
\gamma_{\delta_3}(x) := \inf_{t>0} \left\{ t\,|\, x/t \in \Lambda_{\delta_3,0}\right\}
\end{equation}
is well-defined.
Finally, for $i \in [n]$, define $g^{(i)}_{\delta_3} := -\left(\delta_3(1-b_i)\right)^{-1}e^{(i)}$ and $g^{(n+1)}_{\delta_3} := (\delta_3\sum_{j=1}^n b_j)^{-1}\onebld$.  The following lemma gives us an alternative formulation of $\gamma_{\delta_3}$.
%
%
\begin{lemma}\label{lem:gamma_delta}
Fix $b \in ([0,1)^n \cap \mathbb Q^n)\setminus\{0\}$ and $\delta_3 > 0$.  
The gauge function $\gamma_{\delta_3}$ given by~(\ref{eqn:gamma_delta}) can be written as
\begin{equation}\label{eqn:gauge}
\gamma_{\delta_3}(x) = \max_{i=1}^{n+1} \, \langle g^{(i)}_{\delta_3},x \rangle.
\end{equation}
Furthermore, $\gamma_{\delta_3}$ is sublinear, i.e., $\gamma_{\delta_3}$ is subadditive and $\gamma_{\delta_3}(tx) = t\gamma_{\delta_3}(x)$ for all $t \in \mathbb R_+,x \in \mathbb R^n$.
\end{lemma}
%
%
\begin{proof}
Consider the polar of the closed convex set $\Lambda_{\delta_3,0}$, given by
\[
\Lambda_{\delta_3,0}^\circ := \{ x \in \mathbb R^n \,|\, \langle x,v \rangle \leq 1 \text{ for all } v \in \Lambda_{\delta_3,0}\} 
\]
(c.f.~\cite[Ch.~14]{rock}).  Note that $x \in \Lambda_{\delta_3,0}^\circ$ if and only if $\langle x, v_{\delta_3}^{(i)} \rangle \leq 1$ for all $i \in [n+1]$, i.e., $\langle x, v \rangle \leq 1$ for all extreme points $v$ of $\Lambda_{\delta_3,0}$.  Thus, 
\[
x \in \Lambda_{\delta_3,0}^\circ  \iff \langle x , v_{\delta_3}^{(i)} \rangle \leq 1 \text{ for all } i \in [n+1]  \iff \begin{cases}\, \sum_{j=1}^n (b_j-1) x_j +n x_i \leq \delta_3^{-1} \text{ for all } i \in [n]\\  \sum_{j=1}^n (b_j-1) x_j \leq \delta_3^{-1}.\end{cases}
\]
Hence, $\Lambda_{\delta_3,0}^\circ$ is a compact, convex set, and the vertices of $\Lambda_{\delta_3,0}^\circ$, which can be found when exactly $n$ of the $(n+1)$ of the inequalities in the above display are equalities, are $g_{\delta_3}^{(1)}, \dots, g_{\delta_3}^{(n+1)}$.  

Now, since $\Lambda_{\delta_3,0}$ is a closed, convex set containing the origin, by~\cite[Theorem 14.5]{rock}, we have that the gauge function $\gamma_{\delta_3}$ is equal to the support function of  $\Lambda_{\delta_3,0}^\circ$, i.e., 
\[
\gamma_{\delta_3}(x) = \sup_{g \in \Lambda_{\delta_3,0}^\circ} \langle g, x \rangle = \max_{i=1}^{n+1} \, \langle g^{(i)}_{\delta_3},x \rangle,
\]
as the supremum in the above display will be obtained at one of the vertices of $\Lambda_{\delta_3,0}^\circ$.

The sublinearity of $\gamma_{\delta_3}$ follows immediately from \cite[Theorem 5.7]{guler2010foundations}.
\end{proof}
%
%
\subsection{The Perturbation Direction $\pi_{\delta_3}$}\label{subsec:pi_delta}
%
%
We now construct a minimal function $\pi_{\delta_3}$ using the gauge function $\gamma_{\delta_3}$.  We will perturb $\wt \pi$ (c.f.~Proposition~\ref{prop:pwl}) in the direction of $\pi_{\delta_3}$ in order to obtain $\pi_{comb}$ in the next subsection.

Fix $b \in ([0,1)^n \cap \mathbb Q^n)\setminus\{0\}$.  Choose $\delta_3 > 0$ such that $\delta_3 = p^{-1} < \frac{1}{2}$ for some $p \in \mathbb N$, $\Lambda_{2\delta_3,z^{(1)}} \cap \big( b-\Lambda_{2\delta_3,z^{(2)}} \big) = \emptyset$ for all $z^{(1)}, z^{(2)} \in \mathbb Z^n$, and $b \in U_{\delta_3}$.  
Define $\pi_{\delta_3}:\mathbb R^n \rightarrow \mathbb R$  such that 
\begin{equation}\label{eqn:pi_delta}
\pi_{\delta_3}(x) := 
\begin{cases}
    1/2\cdot\gamma_{\delta_3}(x-z)     & \text{if } x \in \Lambda_{\delta_3,z} \text{ for some } z \in \mathbb Z^n,\\
1 - 1/2\cdot\gamma_{\delta_3}(b-x-z) & \text{if } x \in b-\Lambda_{\delta_3,z} \text{ for some } z \in \mathbb Z^n,\\
\frac{1}{2}   & \text{otherwise},
\end{cases}
\end{equation}
which is well-defined given the choice of $\delta_3$.
%
%
The following lemma gives us the minimality of $\pi_{\delta_3}$, as well as some positive lower bounds on $\Delta \pi_{\delta_3}(x,y)$ for different sets of $(x,y) \in \mathbb R^n \times \mathbb R^n$.  This result will lead to lower bounds on $\Delta \pi_{comb}(x,y)$ is the next subsection.
%
%
\begin{lemma}\label{lem:pi_delta}
Let $b \in ([0,1)^n \cap \mathbb Q^n)\setminus\{0\}$ and $\delta_3 > 0$ such that $\delta_3 = p^{-1} < \frac{1}{2}$ for some $p \in \mathbb N$, $b \in U_{\delta_3}$, and $\Lambda_{2\delta_3,z^{(1)}} \cap \big( b-\Lambda_{2\delta_3,z^{(2)}} \big) = \emptyset$ for all $z^{(1)}, z^{(2)} \in \mathbb Z^n$. Then the function $\pi_{\delta_3}$ defined in~(\ref{eqn:pi_delta}) is piecewise linear and minimal.  Furthermore,
\begin{itemize}
\item[(i)] $\Delta \pi_{\delta_3}(x,y) \geq \frac{1}{2}$, if $x,y \notin \Lambda_{\delta_3,z}$, $x+y \notin b-\Lambda_{\delta_3,z} $ for all $z \in \mathbb Z^n$;
\item[(ii)] $\Delta \pi_{\delta_3}(x,y) \geq \frac{1}{2(n+1)\delta_3}\|x - z^{(1)}\|_\infty$, if $x \in \Lambda_{\delta_3,z^{(1)}}$ for some $z^{(1)} \in \mathbb Z^n$, but $y \notin \Lambda_{\delta_3,z}$, $x+y \notin b-\Lambda_{\delta_3,z}$ for all $z \in \mathbb Z^n$; 
\item[(iii)] $\Delta \pi_{\delta_3}(x,y) \geq \frac{1}{2(n+1)\delta_3}\|b-x-y- z^{(1)}\|_\infty$, if $x+y \in b-\Lambda_{\delta_3,z^{(1)}}$ for some $z^{(1)} \in \mathbb Z^n$, but $x,y \notin \Lambda_{\delta_3,z}$ for all $z \in \mathbb Z^n$.
\end{itemize} 
\end{lemma}
%
%
\begin{proof}
For any $z \in \mathbb Z^n$, note that $1/2\cdot\gamma_{\delta_3}(x - z)  \leq \frac{1}{2}$ if and only if $x \in \Lambda_{\delta_3,z}$ via~(\ref{eqn:S_simplices}) and~(\ref{eqn:gamma_delta}).  
Likewise, $1/2\cdot\gamma_{\delta_3}(b - x - z)  \leq \frac{1}{2}$ if and only if  $x \in b-\Lambda_{\delta_3,z}$.  
Hence, we may write
\begin{equation}\label{eqn:min_pi_delta}
\pi_{\delta_3}(x) = \min_{z \in \mathbb Z^n}\Big[ \min \Big(1/2,\, 1/2\cdot \gamma_{\delta_3}(x-z)\Big)\Big] + \max_{z \in \mathbb Z^n} \Big[  \max \Big(0, 1/2 - 1/2\cdot \gamma_{\delta_3}(b-x-z)\Big)\Big].
\end{equation}
Therefore $\pi_{\delta_3}$ is piecewise linear, as $\pi_{\delta_3}$ is given by the sum of the minimum of some piecewise linear functions and the maximum of some other piecewise linear functions; because $\gamma_{\delta_3}$ is a piecewise linear gauge function, this minimum and this maximum may both be taken over only finitely many piecewise linear functions on every compact set.

Next, we shall establish that $\pi_{\delta_3}$ satisfies (C1) and (C2).  For each $z \in \mathbb Z^n$, we have that $\pi_{\delta_3}(z) = \frac{1}{2}\gamma_{\delta_3}(0) = 0$.  Further, by~(\ref{eqn:min_pi_delta}) and the nonnegativity of $\gamma_{\delta_3}$, we have that $\pi_{\delta_3} \geq 0$, so that $\pi_{\delta_3}$ satisfies condition (C1).  To see that (C2) is satisfied, note that if $x \notin \Lambda_{\delta_3,z},b- \Lambda_{\delta_3,z}$ for all $z \in \mathbb Z^n$, then $\pi_{\delta_3}(x) + \pi_{\delta_3}(b-x) = \frac{1}{2}+\frac{1}{2} =1$.  Otherwise, 
\begin{align*}
&\pi_{\delta_3}(x) + \pi_{\delta_3}(b-x) = 1/2\cdot \gamma_{\delta_3}(x-z) + 1-1/2\cdot\gamma_{\delta_3}(x-z) = 1\;\,\, \qquad\quad\text{ if } \;x \in \Lambda_{\delta_3,z},\\
\text{ and } \quad &\pi_{\delta_3}(x) + \pi_{\delta_3}(b-x) = 1-1/2\cdot \gamma_{\delta_3}(b-x-z) + 1/2\cdot\gamma_{\delta_3}(b-x-z) = 1 \;\text{ if } \;x \in b-\Lambda_{\delta_3,z}.
\end{align*} 
Thus, $\pi_{\delta_3}$ satisfies condition (C2).

Let $x,y \in \R^n$ and consider the following three statements:
\begin{itemize}
\item[(a)] $x \in \Lambda_{\delta_3,z}$ for some $z \in \mathbb Z^n$,
\item[(b)] $y \in \Lambda_{\delta_3,z}$ for some $z \in \mathbb Z^n$, and
\item[(c)] $x+y \in b-\Lambda_{\delta_3,z}$ for some $z \in \mathbb Z^n$.
\end{itemize} 
By evaluating each of the following cases, we will demonstrate that items (i) - (iii) of our result hold and that $\pi_{\delta_3}$ satisfies (C3).

{\bf Case 1:} Suppose that all of statements (a), (b), and (c) are false, or equivalently that $x,y \notin \Lambda_{\delta_3,z}$, $x+y \notin b-\Lambda_{\delta_3,z} $ for all $z \in \mathbb Z^n$.  Then item (i) must hold, as $\pi_{\delta_3}(x), \pi_{\delta_3}(y) \geq \frac{1}{2}$ and $\pi_{\delta_3}(x+y) \leq \frac{1}{2}$.

{\bf Case 2:} Suppose that all of statements (a), (b), and (c) are true.  Then there exist $z^{(1)}, z^{(2)}, z^{(3)} \in \mathbb Z^n$ such that $x \in \Lambda_{\delta_3,z^{(1)}}$, $y \in \Lambda_{\delta_3,z^{(2)}}$, $x+y \in \big( b-\Lambda_{\delta_3,z^{(3)}}\big)$.  But this implies that $x+y \in \Lambda_{2\delta_3,(z^{(1)}+z^{(2)})}$ so that $\Lambda_{2\delta_3,(z^{(1)}+z^{(2)})} \cap \big(b-\Lambda_{2\delta_3,z^{(3)}} \big)\ne \emptyset$, a contradiction of the choice of $\delta_3$.

{\bf Case 3:} Suppose that exactly one of (a) or (b) is true, and (c) is false;  without loss of generality, take (a) to be true and (b) to be false.  Equivalently, we have $x \in \Lambda_{\delta_3,z^{(1)}}$ for some $z^{(1)} \in \mathbb Z^n$, but $y \notin \Lambda_{\delta_3,z}$, $x+y \notin b-\Lambda_{\delta_3,z}$ for all $z \in \mathbb Z^n$.    
Thus, $\pi_{\delta_3}(y) \geq \frac{1}{2} \geq \pi_{\delta_3}(x+y)$.  Furthermore, \begin{equation}\label{eqn:Lambda_subset}\Lambda_{\delta_3,0} \subseteq \bar B_{(n+1)\delta_3}^\infty(0),\end{equation} since 
\begin{align}
r \in \Lambda_{\delta_3,0} & \implies r = \sum_{i=1}^{n+1} \lambda_i v_{\delta_3}^{(i)} \text{ with each } \lambda_i \geq 0 \text{ and } \sum_{i=1}^{n+1} \lambda_i = 1 \notag\\
				     & \implies \| r\|_\infty  =  \Big\| \sum_{i=1}^n \lambda_i \delta_3 (b-\onebld + n e^{(i)}) + \lambda_{n+1}(b-\onebld)\Big\|_\infty \notag \\
				     & \qquad \qquad \quad \leq \delta_3\sum_{i=1}^{n+1} \lambda_i \|b-\onebld\|_\infty + n\delta_3\Big\|\sum_{i=1}^n \lambda_i e^{(i)}\Big\|_\infty  \leq (n+1)\delta_3.\notag
\end{align}
Hence,
\begin{align*}
\Delta \pi_{\delta_3}(x,y) &\geq \pi_{\delta_3}(x) = 1/2\cdot \gamma_{\delta_3}(x-z^{(1)}) = 1/2\cdot\inf\Big\{t >0 \,\Big|\, t^{-1}(x-z^{(1)}) \in \Lambda_{\delta_3,0}\Big\}\\
 					   &\geq  1/2\cdot\inf \Big\{t >0 \,\Big|\, t^{-1}(x-z^{(1)}) \in \bar B_{(n+1)\delta_3}^\infty(0) \} \Big\} = \frac{1}{2(n+1)\delta_3} \|x -z^{(1)} \|_\infty.
\end{align*}
This demonstrates that item (ii) holds.

Before continuing, observe that by~(\ref{eqn:Lambda_subset}),
\begin{equation}\label{eqn:case_3}
 \Lambda_{\delta_3,z} \subseteq \bar B_{(n+1)\delta_3}^\infty(z) \quad \text{ for all } \quad z \in \mathbb Z^n,
\end{equation}
which we will use later on.

{\bf Case 4:} Suppose that (a) and (b) are false, but (c) is true, i.e., $x+y \in b-\Lambda_{\delta_3,z^{(1)}}$ for some $z^{(1)} \in \mathbb Z^n$, but $x,y \notin \Lambda_{\delta_3,z}$ for all $z \in \mathbb Z^n$.  Then $\pi_{\delta_3}(x), \pi_{\delta_3}(y) \geq \frac{1}{2}$, and by an argument similar to that in Case 3,
\[
\Delta \pi_{\delta_3}(x,y) \geq 1- \big(1- 1/2\cdot\gamma_{\delta_3}(b-x-y-z^{(1)})\big) = 1/2\cdot\gamma_{\delta_3}(b-x-y-z^{(1)}) \geq \frac{1}{2(n+1)\delta_3} \|b-x-y-z^{(1)}\|_\infty.
\]
This verifies item (iii).

{\bf Case 5:}  Suppose that (a) and (b) are true, but (c) is false.  Then there exist $z^{(1)}, z^{(2)} \in \mathbb Z^n$ such that $x \in \Lambda_{\delta_3,z^{(1)}}$ and $y \in \Lambda_{\delta_3,z^{(2)}}$.  In view of the subadditivity of $\gamma_{\delta_3}$ (c.f.~Lemma~\ref{lem:gamma_delta}), the fact that $x+y \notin b-\Lambda_{\delta_3,z}$ for all $z \in \mathbb Z^n$ (so that $1/2\cdot \gamma_{\delta_3}(b-x-y-z) \geq 1/2$ for all $z \in \mathbb R^n$ and $\frac12 \geq \pi_{\delta_3}(x+y)$),
\[
\pi_{\delta_3}(x) + \pi_{\delta_3}(y) = 1/2\cdot\gamma_{\delta_3}(x-z^{(1)}) + 1/2\cdot\gamma_{\delta_3}(y-z^{(2)}) \geq 1/2\cdot\gamma_{\delta_3}(x+y-z^{(1)}-z^{(2)}) \geq \pi_{\delta_3}(x+y).
\]
Hence, $\Delta \pi_{\delta_3}(x,y) \geq 0$.

{\bf Case 6:}  Suppose that exactly one of (a) or (b) is true, and (c) is true.  Without loss of generality, take (a) to be true and (b) to be false, so that $x \in \Lambda_{\delta_3,z^{(1)}}$ and $x+y \in b-\Lambda_{\delta_3,z^{(2)}}$ for some $z^{(1)}, z^{(2)} \in \mathbb Z^n$.  Since, $y \notin \Lambda_{\delta_3,z}$ for all $z \in \mathbb Z^n$, we have that $1/2 \leq 1/2\cdot\gamma_{\delta_3}(y-z)$ for all such $z$.  
Thus, by~(\ref{eqn:min_pi_delta}), $\pi_{\delta_3}(y) \geq 1 - 1/2\cdot\gamma_{\delta_3}(b-y-z^{(2)}-z^{(1)})$.  Therefore,
\begin{align*}
\Delta \pi_{\delta_3}(x,y) & = \pi_{\delta_3}(x) + \pi_{\delta_3}(y) - \pi_{\delta_3}(x+y) \\
&\geq 1/2\cdot\gamma_{\delta_3}(x-z^{(1)}) + 1 - 1/2\cdot\gamma_{\delta_3}(b-y-z^{(2)}-z^{(1)}) - \big( 1- 1/2\cdot\gamma_{\delta_3}(b-x-y-z^{(2)}) \big)\\
& = 1/2\cdot\gamma_{\delta_3}(x-z^{(1)}) + 1/2\cdot\gamma_{\delta_3}(b-x-y-z^{(2)}) - 1/2\cdot\gamma_{\delta_3}(b-y-z^{(2)}-z^{(1)})
 \geq 0,
\end{align*}
where the last inequality follows from the subadditivity of $\gamma_{\delta_3}$ (c.f.~Lemma~\ref{lem:gamma_delta}).

Combining Cases 1-6 demonstrate that $\pi_{\delta_3}$ satisfies (C3), and is thus  minimal.  This completes the proof.
\end{proof}
%
%
\subsection{Reintroducing Subadditivity}\label{subsec:pi_comb}
%
%
By taking a convex combination of $\wt \pi$ (c.f.~Proposition~\ref{prop:pwl}) and $\pi_{\delta_3}$ (for $\delta_3$ sufficiently small), we now construct $\pi_{comb}$, a piecewise linear,  minimal approximation of the continuous,  minimal function $\pi$.   We shall see that not only $\pi_{comb}$ is subadditive, but also that $\Delta \pi_{comb}(x,y) > 0$ for ``most" $(x,y) \in \mathbb R^n \times \mathbb R^n$.
%
%

Given $\epsilon > 0$ and a continuous,  minimal $\pi:\mathbb R^n \rightarrow \mathbb R$ for $b \in ([0,1)^n \cap \mathbb Q^n)\setminus\{0\}$, let $\wt\pi$ be any function that satisfies the conclusions of Proposition~\ref{prop:pwl} for suitable $\wt \delta > 0$. 
Let $\wt L$ denote the Lipschitz constant of $\wt\pi$ with respect to the $\infty$-norm (c.f.~(ii) of Proposition~\ref{prop:pwl}).  Then
choose $\delta_3 > 0$ such that 
\begin{equation}\label{eqn:delta_3}
\delta_3 = p^{-1} \leq \min\left\{ \frac{5\epsilon}{12\wt L(n+1)},  \frac{\wt\delta}{2(n+1)}, \frac{5 \wt\delta \epsilon}{12(n+1)} \right\} \text{ for some } p \in \mathbb N \text{ and } b \in U_{\delta_3}.
\end{equation}
Define the function $\pi_{comb}:\mathbb R^n \rightarrow \mathbb R$ via
\begin{equation}\label{eqn:pi_comb}
\pi_{comb} := \Big(1-\frac{5\epsilon}{6}\Big)\wt\pi + \frac{5\epsilon}{6} \pi_{\delta_3}.
\end{equation}
In the following proposition, we demonstrate that $\pi_{comb}$ is  minimal.  We also derive lower bounds on $\Delta \pi_{comb}(x,y)$, for certain $(x,y) \in \mathbb R^n \times \mathbb R^n$.  
%
%
\begin{proposition}\label{prop:pi_comb}
Fix $b \in ([0,1)^n \cap \mathbb Q^n)\setminus\{0\}$, $\epsilon > 0$, and a continuous,  minimal function $\pi:\mathbb R^n \rightarrow \mathbb R$. 
Let $\wt \pi:\mathbb R^n \rightarrow \mathbb R$ be any function that satisfies the conclusions of Proposition~\ref{prop:pwl} for $\wt\delta > 0$, and choose $\delta_3 > 0$ satisfying~\eqref{eqn:delta_3}. 
Let $\pi_{comb}$ be the piecewise linear function given by~(\ref{eqn:pi_comb}).  Then, 
(i) $\|\wt\pi - \pi_{comb}\|_\infty \leq \frac{5\epsilon}{6}$, 
(ii) $\pi_{comb}$ is  minimal, 
(iii) if $x, y \notin \Lambda_{\delta_3,z}$, $x+y \notin b-\Lambda_{\delta_3,z}$ for all $z \in \mathbb Z^n$, then $\Delta \pi_{comb}(x,y) \geq \frac{\epsilon}{4}$, and 
(iv) $\frac{5\epsilon}{12}\gamma_{\delta_3}(x)\geq \pi_{comb}(x)$ for all $x \in \mathbb R^n$.
\end{proposition}
%
%
\begin{proof}
(i) 
To see that (i) holds, note that by Proposition~\ref{prop:pwl} and Lemma~\ref{lem:pi_delta}, both $\wt\pi$ and $\pi_{\delta_3}$ satisfy (C1) and (C2).  Hence, $0 \leq \wt\pi, \pi_{\delta_3} \leq 1$, and 
\[
\| \wt\pi - \pi_{comb}\|_\infty = \frac{5\epsilon}{6} \|\wt\pi - \pi_{\delta_3}\|_\infty \leq \frac{5\epsilon}{6}.
\]  

(ii) and (iii)
To see that $\pi_{comb}$ satisfies conditions (C1) and (C2), note that $\pi_{comb}$ is a convex combination of $\wt\pi$ and $\pi_{\delta_3}$, both of which satisfy these conditions by Proposition~\ref{prop:pwl} and Lemma~\ref{lem:pi_delta} respectively.  Therefore, $\pi_{comb}$ satisfies conditions (C1) and (C2) as well.

Let $x,y \in \R^n$ and consider the following statements (just as in the proof of Lemma~\ref{lem:pi_delta}):
\begin{itemize}
\item[(a)] $x \in \Lambda_{\delta_3,z}$ for some $z \in \mathbb Z^n$,
\item[(b)] $y \in \Lambda_{\delta_3,z}$ for some $z \in \mathbb Z^n$, and
\item[(c)] $x+y \in b-\Lambda_{\delta_3,z}$ for some $z \in \mathbb Z^n$.
\end{itemize} 
In order to establish (ii) and (iii), we therefore consider each of the following five cases.  Note that (a), (b), and (c) cannot all be true (c.f.~Case 2 from the proof of Lemma~\ref{lem:pi_delta}).  Also observe that for each $z \in \mathbb Z^n$, we have
\begin{equation}\label{eqn:subset}
\Lambda_{\delta_3,z} \subseteq \bar B^\infty_{\wt\delta/2}(z) \subseteq \bar B^\infty_{\wt\delta}(z)
\end{equation}
via~(\ref{eqn:case_3}) and the choice of $\delta_3 \leq \frac{\wt\delta}{2(n+1)}$ in~(\ref{eqn:delta_3}).

{\bf Case 1:} Suppose that all of statements (a), (b), and (c) are false so that $x, y \notin \Lambda_{\delta_3,z}$, $x+y \notin b-\Lambda_{\delta_3,z}$ for all $z \in \mathbb Z^n$.  Hence, by~(\ref{eqn:pi_comb}), (i) of Lemma~\ref{lem:pi_delta}, and (iv) of Proposition~\ref{prop:pwl}, 
\[
\Delta \pi_{comb}(x,y) = \frac{5\epsilon}{6} \Delta \pi_{\delta_3}(x,y) + \Big(1-\frac{5\epsilon}{6}\Big) \Delta \wt\pi(x,y) \geq \frac{5\epsilon}{12} - \frac{\epsilon}{6} = \frac{\epsilon}{4}.
\]
This demonstrates that (iii) holds.

{\bf Case 2:} Suppose that exactly one of (a) or (b) is true, and (c) is false.  Without loss of generality, take (a) to be true and (b) to be false, so that $x \in \Lambda_{\delta_3,z^{(1)}} \subseteq \bar B^\infty_{\wt\delta}(z^{(1)})$ (via~(\ref{eqn:subset})) for some $z^{(1)} \in \mathbb Z^n$.  Therefore, 
\begin{align*}
\Delta \pi_{comb}(x,y) & = \frac{5\epsilon}{6} \Delta \pi_{\delta_3}(x,y) + \Big(1-\frac{5\epsilon}{6}\Big) \Delta \wt\pi(x,y)\\
					  & \geq \frac{5\epsilon}{12(n+1)\delta_3} \|x -z^{(1)} \|_\infty - |\wt\pi(x) + \wt\pi(y) - \wt\pi(x+y)|\\
					  & = \frac{5\epsilon}{12(n+1)\delta_3} \|x -z^{(1)} \|_\infty - | \wt\pi(y) - \wt\pi(x-z^{(1)}+y)|\\
					  & \geq \frac{5\epsilon}{12(n+1)\delta_3} \|x -z^{(1)} \|_\infty - \wt L\, \|x -z^{(1)} \|_\infty 
					   \geq 0,
\end{align*}
where the first equality follows from~\eqref{eqn:pi_comb}, the first inequality follows from (ii) of Lemma~\ref{lem:pi_delta}, the second equality follows from~(vii) and (vi) of Proposition~\ref{prop:pwl}, the second inequality follows from the Lipschitz continuity of $\wt\pi$, and last inequality follows from the choice of $\delta_3 \leq \frac{5\epsilon}{12\wt L(n+1)}$ in~(\ref{eqn:delta_3}).

{\bf Case 3:} Suppose that (a) and (b) are false, but (c) is true, so that $x+y \in b - \Lambda_{\delta_3,z^{(1)}}\subseteq  \bar B^\infty_{\wt\delta}(b-z^{(1)})$  (via~(\ref{eqn:subset})) for some $z^{(1)} \in \mathbb Z^n$.  
Then, 
\begin{align*}
\Delta \pi_{comb}(x,y) & = \frac{5\epsilon}{6} \Delta \pi_{\delta_3}(x,y) + \Big(1-\frac{5\epsilon}{6}\Big) \Delta \wt\pi(x,y)\\
					  & \geq \frac{5\epsilon}{12(n+1)\delta_3} \|b-x-y-z^{(1)} \|_\infty - |\wt\pi(x) + \wt\pi(y) - \wt\pi(x+y)|\\
					  & = \frac{5\epsilon}{12(n+1)\delta_3} \|b-x-y-z^{(1)} \|_\infty - |1-\wt\pi(b-x) + \wt\pi(y) - \wt\pi(x+y)| \\
					  & = \frac{5\epsilon}{12(n+1)\delta_3} \|b-x-y-z^{(1)} \|_\infty - |\wt\pi(y+z^{(1)}) - \wt\pi(b-x)| \\
					  & \geq \frac{5\epsilon}{12(n+1)\delta_3}  \|b-x-y-z^{(1)} \|_\infty  - \wt L\,\|b-x-y-z^{(1)}\|_\infty  \geq 0,
\end{align*}
where the first equality follows from~(\ref{eqn:pi_comb}), the first inequality follows from (iii) of Lemma~\ref{lem:pi_delta}, the second and third equalities follow from (v), (vi), and (viii) of Proposition~\ref{prop:pwl}, and the last two inequalities follow from the Lipschitz continuity of $\wt\pi$ and the choice of $\delta_3$.

{\bf Case 4:}  Suppose that (a) and (b) are true, but (c) is false.  Then there exist $z^{(1)}, z^{(2)} \in \mathbb Z^n$ such that $x \in \Lambda_{\delta_3,z^{(1)}} \subseteq \bar B^\infty_{\wt\delta/2}(z^{(1)})$ and $y \in \Lambda_{\delta_3,z^{(2)}} \subseteq \bar B^\infty_{\wt\delta/2}(z^{(2)})$ (c.f.~(\ref{eqn:subset})).  Hence, we must have that $x+y \in\bar B^\infty_{\wt\delta}(z^{(1)}+z^{(2)})$.  Therefore, by (vii) of Proposition~\ref{prop:pwl}, we have $\Delta \wt\pi(x,y) = 0$.  Since $\pi_{\delta_3}$ is subadditive via Lemma~\ref{lem:pi_delta}, we obtain $\Delta \pi_{comb}(x,y) \geq 0$ from~(\ref{eqn:pi_comb}).

{\bf Case 5:}  Suppose that exactly one of (a) or (b) is true, and (c) is true.  Without loss of generality, take (a) to be true and (b) to be false, so that $x \in \Lambda_{\delta_3,z^{(1)}} \subseteq \bar B^\infty_{\wt\delta/2}(z^{(1)})$ and $x+y \in b-\Lambda_{\delta_3,z^{(2)}} \subseteq \bar B^\infty_{\wt\delta/2}(b-z^{(2)})$ for some $z^{(1)}, z^{(2)} \in \mathbb Z^n$ (c.f.~(\ref{eqn:subset})).  We must then have that $y \in \bar B^\infty_{\wt\delta}(b-z^{(2)}-z^{(1)})$, and thus $\Delta \wt\pi(x,y) = 0$ via (vii) and (viii) of Proposition~\ref{prop:pwl}.  It follows that $\Delta \pi_{comb}(x,y) \geq 0$ by the subadditivity of $\pi_{\delta_3}$ (c.f.~Lemma~\ref{lem:pi_delta} and~(\ref{eqn:pi_comb})).   

The above five cases and the preceding discussion demonstrate that (ii) holds.

(iv) To see that (iv) holds, first note that 
\begin{equation}\label{eqn:half_bound}
1/2\cdot \gamma_{\delta_3}(b-x-z) \geq 1/2 \;\, \text{ for all } \; z \in \mathbb Z^n \; \text{ and } \; x \in \bar B^\infty_{\wt\delta}(0), 
\end{equation}
as otherwise, we have by~(\ref{eqn:gamma_delta}) and~(\ref{eqn:subset}) that
\[
b-x-z \in \Lambda_{\delta_3,0} \subseteq  \bar B^\infty_{\wt\delta}(0)
\implies x \in  \bar B^\infty_{\wt\delta}(b-z)
\implies  \bar B^\infty_{\wt\delta}(0) \cap \bar B^\infty_{\wt\delta}(b-z) \ne \emptyset,
\]
a contradiction of (ix) in Proposition~\ref{prop:pwl}.  Now, for all $x \in  \bar B^\infty_{\wt\delta}(0)$, 
\[
\pi_{comb}(x) = \frac{5\epsilon}{6} \pi_{\delta_3}(x) \leq \frac{5\epsilon}{12} \gamma_{\delta_3}(x),
\]
where the equality follows from~(\ref{eqn:pi_comb}) and the fact that $\tilde \pi(x) = 0$ by (vii) of Proposition~\ref{prop:pwl}; the inequality follows from~(\ref{eqn:min_pi_delta}), in which the second term is zero by~(\ref{eqn:half_bound}).

 Next, observe that $\Big( \frac{\wt\delta}{(n+1)\delta_3}\cdot \Lambda_{\delta_3,0}\Big) \subseteq \bar B^\infty_{\wt\delta}(0)$ by~(\ref{eqn:case_3}).  Therefore, for all $x \notin  \bar B^\infty_{\wt\delta}(0)$, we have $\frac{(n+1)\delta_3}{\wt\delta} x \notin  \Lambda_{\delta_3,0}$ and
\[
\pi_{comb}(x) \leq 1 \leq \frac{5\wt\delta\epsilon}{12\delta_3 (n+1)} \leq \frac{5\epsilon}{12} \gamma_{\delta_3}(x), 
\]
via the fact that $0 \leq \pi_{comb} \leq 1$ (as $\pi_{comb}$ is  minimal) and the choice of $\delta_3 \leq  \frac{5 \wt\delta \epsilon}{12(n+1)} $~(c.f.~(\ref{eqn:delta_3})).  
Hence, statement (iv) holds as well.
\end{proof}
%
%
\paragraph{Piecewise linear minimal functions are dense in the set of continuous minimal functions.} In~\cite[Lemma 1]{basu2016minimal}, the authors demonstrate that when $n =1$, for any $\epsilon > 0$, a continuous,  minimal function $\pi:\mathbb R^n \rightarrow \mathbb R$ can be approximated by a piecewise linear,  minimal function $\ol \pi:\mathbb R^n \rightarrow \mathbb R$ such that $\|\pi-\ol\pi\|_\infty < \epsilon$.  
In the proof of this result, $\ol\pi$ is the interpolant of $\pi$ whose knots are given by $\frac{1}{q}\mathbb Z$ for some appropriately chosen $q \in \mathbb N$ sufficiently large.  The subadditivity of $\ol\pi$ follows from the subadditivity of $\pi$ and the structure of the underlying polyhedral complex of $\Delta \ol \pi$.  
However, when $n > 1$, this particular argument fails, as the underlying polyhedral complex of $\Delta \ol \pi$ is much more difficult to describe.  Hence, new techniques, such as those found in this paper, are required to establish the existence of such a $\ol\pi$.  
The following corollary generalizes~\cite[Lemma 1]{basu2016minimal} to arbitrary $n \in \mathbb N$.  The proof follows directly from Propositions~\ref{prop:pwl} and~\ref{prop:pi_comb}.  
%
%
\begin{corollary}\label{cor:PWL-approx}
Fix $b \in ([0,1)^n \cap \mathbb Q^n)\setminus\{0\}$ and let $\pi:\mathbb R^n \rightarrow \mathbb R$ be a continuous,  minimal function.  Then for each $\epsilon > 0$ there exists a piecewise linear,  minimal function $\ol \pi:\mathbb R^n \rightarrow \mathbb R$ such that $\| \pi- \ol\pi\|_\infty < \epsilon$.
\end{corollary}
%
%
\begin{proof}
For $\epsilon > 0$ there exists $\wt \pi:\mathbb R^n \rightarrow \mathbb R$ such that $\wt \pi$ satisfies the conclusions of Proposition~\ref{prop:pwl}, including $\|\pi - \wt \pi \|_\infty < \frac{\epsilon}{18}$.  Furthermore, by Proposition~\ref{prop:pi_comb}, there exists a piecewise linear,  minimal $\pi_{comb}: \mathbb R^n \rightarrow \mathbb R$ such that $\| \wt\pi - \pi_{comb}\|_\infty \leq \frac{5\epsilon}{6}$.  Taking $\ol \pi \equiv \pi_{comb}$ and applying the triangle inequality completes the proof.
\end{proof}
%
%
\section{Symmetric $(n+1)$-Slope Fill-in}\label{sec:sym_fill_in}
%
%
Our goal is to produce a piecewise linear,  minimal function $\pi^\prime$ that (i) approximates a given continuous, minimal function $\pi$, and (ii) is 
an extreme function of the set of  minimal functions.  
By \cite[Theorem 1.7]{basu-hildebrand-koeppe-molinaro:k+1-slope}, any $(n+1)$-slope (c.f.~Definition~\ref{def:slope}), genuinely $n$-dimensional (c.f.~Definition~\ref{def:genuine}),  minimal function is an extreme function.  Therefore, in this section, we manipulate the piecewise linear,  minimal (and in fact genuinely $n$-dimensional) $\pi_{comb}$ from Section~\ref{sec:perturb} to obtain a $(n+1)$-slope, genuinely $n$-dimensional,  minimal function $\pi^\prime$ (that is therefore extreme).  

We begin by applying a $(n+1)$-slope fill-in procedure (c.f.~(\ref{eqn:fill_in})) to $\pi_{comb}$, in order to obtain a $(n+1)$-slope function $\pi_{fill-in}$.  Although the $(n+1)$-slope fill-in procedure preserves the subadditivity of $\pi_{comb}$, as well as property (C1), it does not preserve the symmetry property, (C2).  To reintroduce property (C2), a symmetrization technique is carefully applied in~(\ref{eqn:pi_sym}), producing $\pi_{sym}$.  We shall see that $\pi_{sym}$ is in fact a genuinely $n$-dimensional, $(n+1)$-slope,  minimal function, which accurately approximates $\pi$.  Thus, $\pi^\prime \equiv \pi_{sym}$ is the desired extreme function.  The results of this section will allow us to complete the proof of Theorem~\ref{thm:main_thm} in Section~\ref{sec:proof}.
%
%
\subsection{A $(n+1)$-Slope Fill-in Procedure}
%
%
In this subsection, we apply the $(n+1)$-slope fill-in procedure to $\pi_{comb}$ in order to obtain the $(n+1)$-slope subadditive function $\pi_{fill-in}$, which satisfies (C1).  The following definition introduces the formal concept of a $k$-slope function for any $k \in \mathbb N$.
%
%
\begin{definition}\label{def:slope}
Suppose that $\theta:\mathbb R^n \rightarrow \mathbb R$ is a piecewise linear function.  If all the gradient vectors $a^P$ (c.f.~Definition~\ref{def:pwl}) of $\theta$ take values in a set of size $k$, then we say that $\theta$ is a {\bf k-slope function}.
\end{definition}
%
%
In \cite{infinite, infinite2, johnson1974group}, Gomory and Johnson describe a 2-slope fill-in procedure that extends a subadditive function $\theta:U \rightarrow \mathbb R$, for some subgroup $U$ of $\mathbb R$, to a subadditive function $\theta_{fill-in}: \mathbb R \rightarrow \mathbb R$ defined on all of $\mathbb R$.  This procedure is also used in~\cite{basu2016minimal} to produce a 2-slope function and prove Theorem~\ref{thm:main_thm} in the special case when $n = 1$.  We now extend this method to a $(n+1)$-slope fill-in procedure for functions $\theta: \mathbb R^n \rightarrow \mathbb R$, for arbitrary $n \in \mathbb N$.

Let  $\gamma : \mathbb R^n \rightarrow \mathbb R$ be a sublinear, $(n+1)$-slope function.  Additionally, let $U = \frac{1}{q}\mathbb Z^n$ for some $q \in \mathbb N$ and $\theta: \R^n \rightarrow \mathbb R$ be a subadditive function.  Then the $(n+1)$-slope fill-in of the function $\theta$ with respect to $U$ and $\gamma$ is $\theta_{fill-in}:\mathbb R^n \rightarrow \mathbb R$, which is defined such that
\begin{equation}\label{eqn:fill_in}
\theta_{fill-in}(x) := \min_{u \in U} \{ \theta(u) + \gamma(x-u) \}.
\end{equation}
The following lemma discusses several properties of the $(n+1)$-slope fill-in function.
%
%
\begin{lemma}\label{lem:fill_in}[Johnson (Section 7 in~\cite{johnson1974group})]
Let $U = \frac{1}{q}\mathbb Z^n$ for some $q \in \mathbb N$ and $\theta:\mathbb R^n\rightarrow \mathbb R$ be a subadditive function.  Suppose that $\gamma: \mathbb R^n \rightarrow \mathbb R$ is a sublinear, $(n+1)$-slope function with $\theta \leq \gamma$.  Then the $(n+1)$-slope fill-in $\theta_{fill-in}$ of $\theta$ with respect to $U$ and $\gamma$ is subadditive. Furthermore, $\theta_{fill-in} \geq \theta$ and $\theta_{fill-in}|_U =\theta|_U$.  
Finally, $\theta_{fill-in}$ is a $(n+1)$-slope function.
\end{lemma}
%
%
\begin{proof}
Since $\theta$ is subadditive and $\gamma$ is sublinear, we first have that for some $u_1, u_2 \in U$,
\begin{align*}
\theta_{fill-in}(x)+\theta_{fill-in}(y) &=   \theta(u_1) + \gamma(x-u_1) + \theta(u_2) + \gamma(y-u_2)\\
							     &\geq \theta(u_1+u_2) +\gamma(x+y -(u_1+u_2))\\
							     & \geq \theta_{fill-in}(x+y),
\end{align*}
so $\theta_{fill-in}$ is subadditive.  Also, note that for all $u \in U$,
\[
\theta_{fill-in}(u) \leq \theta(u) + \gamma(0) = \theta(u),
\]
and for all $x \in \mathbb R^n$,
\[
\theta_{fill-in}(x) = \min_{u\in U}( \theta(u) + \gamma(x-u)) \geq \min_{u\in U}( \theta(u) + \theta(x-u)) \geq \theta(x),
\]
where the last inequality follows from subadditivity of $\theta$. Thus, $\theta_{fill-in} \geq \theta$ and $\theta_{fill-in}|_U =\theta|_U$.   
By construction, $\theta_{fill-in}$ is a $(n+1)$-slope function.  
\end{proof}
%
%
Fix $\epsilon > 0$, $b \in ([0,1)^n \cap \mathbb Q^n)\setminus\{0\}$, a continuous,  minimal function $\pi$, and suitable $\delta_3 > 0$ satisfying~(\ref{eqn:delta_3}).  Define $L_\gamma := \frac{5\epsilon}{12}\max_{i=1}^{n+1} \| g^{(i)}_{\delta_3}\|_1$ (c.f.~(\ref{eqn:gauge})).  
Choose $\delta_4 > 0$ such that 
\begin{equation}\label{eqn:delta_4}
\delta_4 = \frac{\delta_3}{q} \leq \frac{\epsilon}{36}\Big(2 L_\gamma + \wt L \Big)^{-1} \quad \text{ for some }\quad q \in \mathbb N,
\end{equation} 
where $\wt L$ denotes the Lipschitz constant of some $\wt \pi$ satisfying the conclusions of Proposition~\ref{prop:pwl}.  
Note that we then have $b \in U_{\delta_3} \subseteq  U_{\delta_4}$.  Consider the $(n+1)$-slope fill-in $\pi_{fill-in}:\mathbb R^n \rightarrow \mathbb R$ of $\pi_{comb}$ (c.f.~(\ref{eqn:pi_comb})) with respect to $U_{\delta_4}$ and $\frac{5\epsilon}{12}\gamma_{\delta_3}$:
\begin{equation}\label{eqn:pi_fill_in}
\pi_{fill-in}(x) := \min_{u \in U_{\delta_4}}\Big\{ \pi_{comb}(u) + \frac{5\epsilon}{12}\gamma_{\delta_3}(x-u)\Big\}.
 \end{equation}

We conclude this subsection with the following lemma, in which we collect a number of facts on $\pi_{fill-in}$ to be used in the proof of subsequent results.
%
%
\begin{lemma}\label{lem:pi_fill_in}
Fix $\epsilon > 0$, $b \in ([0,1)^n \cap \mathbb Q^n)\setminus\{0\}$, and a continuous,  minimal function $\pi$.  Let $\wt\pi$ be any function that meets the conclusions of Proposition~\ref{prop:pwl}
and choose suitable $\delta_3,\delta_4 > 0$ satisfying~(\ref{eqn:delta_3}) and~(\ref{eqn:delta_4}).  Construct $\pi_{comb}$ via~(\ref{eqn:pi_comb}).  
Then the function $\pi_{fill-in}$ defined in~(\ref{eqn:pi_fill_in}) 
(i) is a $(n+1)$-slope, subadditive function that satisfies (C1) and $\pi_{fill-in} \geq \pi_{comb}$,
(ii) satisfies $\|\pi_{comb} - \pi_{fill-in} \|_\infty \leq \frac{\epsilon}{36}$, and 
(iii) satisfies $\pi_{fill-in}(x) = \pi_{comb}(x) =\frac{5\epsilon}{12} \gamma_{\delta_3}(x-z)$ if $x \in \Lambda_{\delta_3,z}$ for some $z \in \mathbb Z^n$.
\end{lemma}
%
%
\begin{proof}
(i)  Observe that $\pi_{comb}$ is subadditive as $\pi_{comb}$ is  minimal (c.f.~(ii) of Proposition~\ref{prop:pi_comb}), $\frac{5\epsilon}{12}\gamma_{\delta_3}$ is a sublinear, $(n+1)$-slope function (c.f.~(\ref{eqn:gamma_delta}) and~Lemma~\ref{lem:gamma_delta}), $U_{\delta_4} = \frac{1}{pq} \mathbb Z^n$ for some $pq \in \mathbb N$ (c.f.~(\ref{eqn:delta_3}) and~(\ref{eqn:delta_4})), and $\frac{5\epsilon}{12}\gamma_{\delta_3}(x) \geq \pi_{comb}(x)$ for all $x \in \mathbb R^n$ (c.f.~(iv) of Proposition~\ref{prop:pi_comb}).  By Lemma~\ref{lem:fill_in}, $\pi_{fill-in}$ is a $(n+1)$-slope, subadditive function such that $\pi_{fill-in} \geq \pi_{comb}$.  To see that $\pi_{fill-in}$ satisfies (C1), note that $\pi_{fill-in}(z) = 0$ for each $z \in \mathbb Z^n \subseteq U_{\delta_4}$ and $0 \leq \pi_{comb} \leq \pi_{fill-in}$, both by Lemma~\ref{lem:fill_in}.

(ii)  
By (ii) of Proposition~\ref{prop:pwl},~(\ref{eqn:gauge}),~(\ref{eqn:pi_comb}), and  \cite[Proposition 2.2.7]{scholtes2012introduction}, $\pi_{comb}$ is Lipschitz continuous, with Lipschitz constant $(L_\gamma + \wt L)$, with respect to the $\infty$-norm.  
Similarly, by~(\ref{eqn:gauge}),~(\ref{eqn:pi_fill_in}), and  \cite[Proposition 2.2.7]{scholtes2012introduction}, $\pi_{fill-in}$ is Lipschitz continuous with Lipschitz constant $L_\gamma$, with respect to the $\infty$-norm.    
Therefore, for $x \in \mathbb R^n$, we may choose $u \in U_{\delta_4}$ such that by Lemma~\ref{lem:fill_in} and~(\ref{eqn:delta_4}),
\begin{align*}
&|\pi_{comb}(x) - \pi_{fill-in}(x)| \\
&\qquad\qquad \leq |\pi_{comb}(x) - \pi_{comb}(u)| + |\pi_{comb}(u) - \pi_{fill-in}(u)| + |\pi_{fill-in}(u) - \pi_{fill-in}(x)|\\
&\qquad\qquad \leq  (L_\gamma + \wt L) \|x-u\|_\infty + L_\gamma \|u-x \|_\infty \leq \Big(2 L_\gamma + \wt L \Big) \delta_4 \leq \frac{\epsilon}{36}.
\end{align*}

(iii)
By Lemma~\ref{lem:fill_in}, we have that $\pi_{comb}(x) \leq \pi_{fill-in}(x)$ for all $x \in \mathbb R^n$.  Suppose that $x \in \Lambda_{\delta_3,z}$ for some $z \in \mathbb Z^n$.  
In view of~(\ref{eqn:case_3}),~(\ref{eqn:delta_3}), and (vii) of Proposition~\ref{prop:pwl}, $\wt \pi(x) = 0$.
%
Observe that
\begin{align*}
\pi_{fill-in}(x) & \leq  \pi_{comb}(z) + \frac{5\epsilon}{12} \gamma_{\delta_3}(x-z)  \;\; =  0 + \frac{5\epsilon}{6} \cdot \frac12 \gamma_{\delta_3}(x-z) \\
& =  \Big(1-\frac{5\epsilon}{6}\Big) \wt\pi(x) +  \frac{5\epsilon}{6} \pi_{\delta_3}(x)     = \pi_{comb}(x),
\end{align*}
where the first inequality follows from~(\ref{eqn:pi_fill_in}), the first equality follows from the fact that $\pi_{comb}$ is minimal by (ii) of Proposition~\ref{prop:pi_comb} (implying that $\pi_{comb}(z)=0$), the second equality follows from the fact that $\wt \pi(x) = 0$ and~(\ref{eqn:pi_delta}), and the last equality follows from~(\ref{eqn:pi_comb}). Thus, $\pi_{fill-in}(x) = \pi_{comb}(x) = \frac{5\epsilon}{6} \pi_{\delta_3}(x) = \frac{5\epsilon}{12} \gamma_{\delta_3}(x-z)$.
\end{proof}
%
%
\subsection{Symmetrization}
%
%
Although $\pi_{fill-in}: \mathbb R^n \rightarrow \mathbb R$ is subadditive, satisfies (C1), and is a $(n+1)$-slope function, the $(n+1)$-slope fill-in procedure used to form $\pi_{fill-in}$ in~(\ref{eqn:pi_fill_in}) likely does not preserve the symmetry property (C2) held by $\pi_{comb}$.  We must therefore make use of additional constructions to symmetrize $\pi_{fill-in}$, just as the authors do in~\cite{basu2016minimal}, in the special case when $n=1$.  
Moreover, in~\cite{basu2016minimal},  the authors choose a set $B \subseteq \mathbb R$ such that, roughly, $B$ and $b-B$ form a partition of $\mathbb R$; to reintroduce symmetry, they form $\pi_{sym}:\mathbb R \rightarrow \mathbb R$ such that $\pi_{sym}(x) = \pi_{fill-in}(x)$ for all $x \in B$ and $\pi_{sym}(x) = 1-\pi_{fill-in}(b-x)$ for all $x\in b-B$.  Because of the relatively simple nature of the underlying polyhedral complex of $\pi_{fill-in}$ when $n = 1$, it is not hard to find $B$ such that the function $\pi_{sym}$ inherits the piecewise linearity and continuity of $\pi_{fill-in}$.  However, the more complicated geometry of the underlying polyhedral complex of $\pi_{fill-in}$ when $n >1$ makes the selection of such a set $B$ much more difficult.  Additionally, even once $B$ is selected, the construction of $\pi_{sym}$ is more elaborate when $n > 1$.  

To aid us with the selection of the set $B$ and the construction of $\pi_{sym}$ when $n > 1$, we assemble two additional functions: $\pi_{aux}:\mathbb R^n \rightarrow \mathbb R$, a  minimal, $(n+1)$-slope function with gradients $\delta_3 g_{\delta_3}^{(1)},\dots,\delta_3 g_{\delta_3}^{(n+1)}$ (c.f.~(\ref{eqn:pi_aux})), and $\eta_{aux}:\mathbb R^n \rightarrow \mathbb R$ (c.f.~(\ref{eqn:eta})), which is obtained via scaling and translating $\pi_{aux}$.

Recall the gauge function $\gamma_{\delta_3}$ given by~(\ref{eqn:gamma_delta}) for $\delta_3 > 0$ and construct
$\pi_{aux}:\mathbb R^n \rightarrow \mathbb R$ via 
\begin{equation}\label{eqn:pi_aux}
\pi_{aux}(x) := \delta_3 \cdot \min_{z \in \mathbb Z^n} \gamma_{\delta_3}(x-z).
\end{equation}
%
The next lemma yields several useful properties of $\pi_{aux}$ to be used in the proof of Lemma~\ref{lem:eta}.  
%
%
\begin{lemma}\label{lem:pi_aux}
Fix $\delta_3 > 0$, $b \in ([0,1)^n \cap \mathbb Q^n)\setminus\{0\}$, and suppose that $2 \leq n \in \mathbb N$.  Then the function $\pi_{aux}:\mathbb R^n \rightarrow \mathbb R$ defined in~(\ref{eqn:pi_aux}) is a  minimal function.  Furthermore, $\pi_{aux}$ is a $(n+1)$-slope function with gradients $\delta_3 g_{\delta_3}^{(1)},\,\dots,\,\delta_3 g_{\delta_3}^{(n+1)}$.
\end{lemma}
%
%
\begin{proof} 
Consider the (convex) simplex $K := \text{conv}(\{0,ne^{(1)},\dots,ne^{(n)}\})$.  
Define $\Lambda_{1,0} = K+b-\onebld$, so that $\gamma_1:= \delta_3 \cdot \gamma_{\delta_3}$ is the gauge function for $\Lambda_{1,0}$.  By~\cite[Theorem 12]{ccz}, the function $\pi_{aux}(x) = \min_{z \in \Z^n} \gamma_1(x-z)$ is a $(n+1)$-slope, extreme (and therefore minimal) function with gradients $\delta_3 g_{\delta_3}^{(1)},\,\dots,\,\delta_3 g_{\delta_3}^{(n+1)}$; alternate proofs of this fact were obtained in the subsequent literature by different authors; see the survey~\cite{basu2015geometric}. 
\end{proof}
%
%
We next use $\pi_{aux}$ to construct the function $\eta_{aux}$, which will be utilized in the symmetrization of $\pi_{fill-in}$ (c.f.~(\ref{eqn:pi_sym})).

For $\epsilon > 0$, $b \in ([0,1)^n \cap \mathbb Q^n)\setminus\{0\}$,  and $\delta_3 > 0$, choose $\tau \in \mathbb N$ such that 
\begin{equation}\label{eqn:tau}
\tau^{-1} \leq  \frac{6\delta_3}{5\epsilon}\quad \text{and} \quad \tau b = b+z\quad \text{for some} \quad z \in \mathbb Z^n. 
\end{equation}
To see that such a $\tau$ exists,  let $s$ denote some common multiple of the denominators of the components of $b$ (in whatever form the components are written in, whether in lowest terms or not). 
Then let $\tau = 1 + ms$ for some $m \in \mathbb N$ such that $\frac{1}{\tau} \leq  \frac{6\delta_3}{5\epsilon}$.  We then have $\tau b = b + msb = b +  z$, where $z := msb \in \mathbb Z^n$ by choice of $s$.
Define $\eta_{aux}:\mathbb R^n \rightarrow \mathbb R$ such that 
\begin{equation}\label{eqn:eta}
\eta_{aux}(x) := \frac{1}{2} + \frac{5\epsilon}{24\delta_3 \tau}  -  \frac{5\epsilon}{12\delta_3 \tau} \pi_{aux}(\tau(b-x)).
\end{equation}

In the next lemma, we gather several properties of the function $\eta_{aux}$.
%
%
\begin{lemma}\label{lem:eta}
Fix $\epsilon > 0$, $b \in ([0,1)^n \cap \mathbb Q^n)\setminus\{0\}$, $\delta_3 > 0$, and choose suitable  $\tau \in \mathbb N$ (c.f.~(\ref{eqn:tau})).  
The function $\eta_{aux}:\mathbb R^n \rightarrow \mathbb R$ given by~(\ref{eqn:eta}) is a $(n+1)$-slope function with gradients $\frac{5\epsilon}{12}g_{\delta_3}^{(1)},\dots,\frac{5\epsilon}{12}g_{\delta_3}^{(n+1)}$.  
Furthermore, the function $\eta_{aux}$ satisfies the symmetry property (C2).  Finally, $\eta_{aux}(x) \in [\frac{1}{4},\frac{3}{4}]$ for all $x \in \mathbb R^n$.
\end{lemma}
%
%
\begin{proof}
Note that by virtue of Lemma~\ref{lem:pi_aux},~(\ref{eqn:eta}), and the chain rule, $\eta_{aux}$ is an $(n+1)$-slope function with gradients $\frac{5\epsilon}{12} g_{\delta_3}^{(1)},\dots,\frac{5\epsilon}{12}g_{\delta_3}^{(n+1)}$.  
To see that $\eta_{aux}$ satisfies property (C2), note that by~(\ref{eqn:tau}), 
\begin{align*}
\eta_{aux}(x) + \eta_{aux}(b-x) &= 1 + \frac{5\epsilon}{12\delta_3 \tau}  -  \frac{5\epsilon}{12\delta_3 \tau} \pi_{aux}(\tau(b-x)) - \frac{5\epsilon}{12\delta_3 \tau} \pi_{aux}(\tau x) \\
					     &= 1 + \frac{5\epsilon}{12\delta_3 \tau} - \frac{5\epsilon}{12\delta_3 \tau} \pi_{aux}(b + z -\tau x) - \frac{5\epsilon}{12\delta_3 \tau} \pi_{aux}(\tau x) = 1,
\end{align*}
as $\pi_{aux}$ is  minimal, and therefore is periodic with respect to $\mathbb Z^n$ and satisfies (C2).  
Finally, note for all $x \in \mathbb R^n$, we have that $\pi_{aux}(x) \in [0,1]$, as $\pi_{aux}$ is  minimal, and hence satisfies (C1) and (C2).  Combining this with the choice of $\tau \in \mathbb N$ that satisfies $\frac{1}{\tau} \leq  \frac{6\delta_3}{5\epsilon}$ (c.f.~(\ref{eqn:tau})), we have that $\eta_{aux}(x) \in [\frac{1}{4},\frac{3}{4}]$ for all $x \in \mathbb R^n$.
This completes the proof.
\end{proof}
%
%
We may now use the function $\eta_{aux}$ from~(\ref{eqn:eta}) to symmetrize $\pi_{fill-in}$.  Recall the  minimal  (c.f.~Proposition~\ref{prop:pi_comb}) function $\pi_{comb}$ defined in~(\ref{eqn:pi_comb}). 
Define $\pi_{sym}:\mathbb R^n \rightarrow \mathbb R$ such that for each $x \in \mathbb R^n$,
\begin{equation}\label{eqn:pi_sym}
\pi_{sym}(x) := 
\begin{cases}
\min(\pi_{fill-in}(x) ,\eta_{aux}(x))	    &\text{ if }  \pi_{comb}(x) < \eta_{aux}(x) \\
1 - \min(\pi_{fill-in}(b-x) ,\eta_{aux}(b-x))  &\text{ if }  \pi_{comb}(x) > \eta_{aux}(x) \\
\eta_{aux}(x) &\text{ if }  \pi_{comb}(x) = \eta_{aux}(x).
\end{cases}
\end{equation}
%
%
We will demonstrate that $\pi_{sym}$ is an extreme function in Proposition~\ref{prop:pi_sym}.  However, we first make the following definition.  
%
%
\begin{definition}\cite{basu-hildebrand-koeppe-molinaro:k+1-slope}\label{def:genuine}
A function $\theta:\mathbb R^n \rightarrow \mathbb R$ is {\bf genuinely $n$-dimensional} if there does not exist a function $\varphi: \mathbb R^{n-1}\rightarrow \mathbb R$ and a linear map $A:\mathbb R^n \rightarrow \mathbb R^{n-1} $ such that $\theta = \varphi \circ A$.
\end{definition}
%
%
We are now ready to show that $\pi_{sym}$ is extreme.  
%
%
\begin{proposition}\label{prop:pi_sym}
Fix $\epsilon \in (0,\frac{1}{2}]$, $b \in ([0,1)^n \cap \mathbb Q^n)\setminus\{0\}$, and a continuous,  minimal function $\pi:\mathbb R^n \rightarrow \mathbb R$.  Let $\wt\pi$ be any function that meets the conclusions of Proposition~\ref{prop:pwl} and choose suitable $\delta_3,\delta_4 > 0$ (c.f.~(\ref{eqn:delta_3}) and~(\ref{eqn:delta_4})).  Construct $\pi_{comb}$ via~(\ref{eqn:pi_comb}) and $\pi_{fill-in}$ via~(\ref{eqn:pi_fill_in}).  Choose $\tau \in \mathbb N$ to satisfy~(\ref{eqn:tau}).
Consider the function $\pi_{sym}: \mathbb R^n \rightarrow \mathbb R$ given by~(\ref{eqn:pi_sym}).  We then have that $\|\pi_{fill-in} - \pi_{sym} \|_\infty \leq \epsilon/18$.  Furthermore, $\pi_{sym}$ is (i) a piecewise linear $(n+1)$-slope function, (ii)  minimal, and (iii) genuinely $n$-dimensional.  Hence, $\pi_{sym}$ is an extreme function of the set of  minimal functions by~\cite[Theorem 1.7]{basu-hildebrand-koeppe-molinaro:k+1-slope}.    
\end{proposition}
%
%
\begin{proof}  
We first show that $\|\pi_{fill-in} - \pi_{sym} \|_\infty \leq \epsilon/18$.  If $\pi_{comb}(x) < \eta_{aux}(x)$, then by~(\ref{eqn:pi_sym}), either $\pi_{sym}(x) = \pi_{fill-in}(x)$, or $\pi_{sym}(x) = \eta_{aux}(x)$.  Therefore, by Lemma~\ref{lem:pi_fill_in},
\[
|\pi_{sym}(x) - \pi_{fill-in}(x)| \leq \pi_{fill-in}(x) - \eta_{aux}(x) <  \pi_{fill-in}(x) - \pi_{comb}(x) \leq \frac{\epsilon}{36}.
\]

If $\pi_{comb}(x) > \eta_{aux}(x)$, then $\pi_{comb}(b-x) < \eta_{aux}(b-x)$, as both $\pi_{comb}$ and $\eta_{aux}$ satisfy (C2) (c.f.~Proposition~\ref{prop:pi_comb} and Lemma~\ref{lem:eta}).  Therefore, 
\begin{align*}
|\pi_{sym}(x) - \pi_{fill-in}(x)| &\leq  |\pi_{sym}(x) - \pi_{comb}(x)| + |\pi_{comb}(x) - \pi_{fill-in}(x)|\\
& \leq |1 - \min(\pi_{fill-in}(b-x) ,\eta_{aux}(b-x)) - 1 + \pi_{comb}(b-x)| + \frac{\epsilon}{36}\\
& \leq |\pi_{comb}(b-x) - \pi_{fill-in}(b-x)| + \frac{\epsilon}{36} \leq \frac{\epsilon}{18},
\end{align*}
where the first inequality follows from the triangle inequality, and the second inequality follows from~(\ref{eqn:pi_sym}) and the symmetry of $\pi_{comb}$ (to obtain the first term), as well as (ii) of Lemma~\ref{lem:pi_fill_in} (to obtain the second term); 
the next inequality then follows from the fact that both $\pi_{comb}(b-x) \leq \pi_{fill-in}(b-x)$ (by (i) of Lemma~(\ref{lem:pi_fill_in})) and $\pi_{comb}(b-x) < \eta_{aux}(b-x)$; the last inequality is a consequence of (ii) of Lemma~\ref{lem:pi_fill_in}.


%
Finally, if $\pi_{comb}(x) = \eta_{aux}(x)$, then by~(\ref{eqn:pi_sym}) and (ii) of Lemma~\ref{lem:pi_fill_in},
\[
|\pi_{sym}(x) - \pi_{fill-in}(x)| = |\pi_{comb}(x) - \pi_{fill-in}(x)| \leq \frac{\epsilon}{36}.
\] 
Thus, $\| \pi_{fill-in} - \pi_{sym}\|_\infty \leq \frac{\epsilon}{18}$.

Next, we establish (i), (ii), and (iii).  

(i) 
Let $\Pcal_1$, $\Pcal_2$,  $\Pcal_3$, $\Pcal_4$, and $\Pcal_5$ denote any five polyhedral complices compatible with the piecewise linear functions $p^{(1)}$, $p^{(2)}$, $p^{(3)}$, $\pi_{comb}$, and $\eta_{aux}$ respectively, where $p^{(1)}:\mathbb R^n \rightarrow \mathbb R$, $p^{(2)}:\mathbb R^n \rightarrow \mathbb R$, and $p^{(3)}:\mathbb R^n \rightarrow \mathbb R$ are given by 
\begin{align*}
p^{(1)}(x) &:= \min(\pi_{fill-in}(x) ,\eta_{aux}(x)),\\
p^{(2)}(x) &:= 1- \min(\pi_{fill-in}(b-x),\eta_{aux}(b-x)), \text{ and } \\
p^{(3)}(x) &:= \min(\pi_{comb}(x), \eta_{aux}(x)). 
\end{align*}
These complices must exist via~\cite[Proposition~2.2.3]{scholtes2012introduction}.  
Then define the polyhedral complex 
\[
\Pcal := \Big\{ \bigcap_{i=1}^5 P_i \,\Big|\, P_i \in \Pcal_i, i \in [5] \Big\},
\]
which is a common refinement of $\Pcal_1, \ldots, \Pcal_5$, and therefore compatible with each of $p^{(1)}$, $p^{(2)}$, $p^{(3)}$, $\pi_{comb}$, and $\eta_{aux}$.
Now, let $P \in \Pcal$; we shall show that $\pi_{sym}$ is affine on $P$.  Suppose that there exist $x,x^\prime \in P$ such that $\pi_{comb}(x) < \eta_{aux}(x)$ and $\pi_{comb}(x^\prime) > \eta_{aux}(x^\prime)$.  
By the intermediate value theorem, there must exist $t \in (0,1)$ such that $\ol x := t x + (1-t) x^\prime$ and $\pi_{comb}(\ol x) = \eta_{aux}(\ol x)$.  Since $\eta_{aux}$ and $p^{(3)}$ are affine on $P$, we must have
\[
\eta_{aux}(\ol x) = t\,\eta_{aux}(x) + (1-t)\, \eta_{aux}(x^\prime) > t\, p^{(3)}(x) + (1-t) \, p^{(3)}(x^\prime) = p^{(3)}(\ol x) = \eta_{aux}(\ol x),
\]
a contradiction.
%
Hence, either $\pi_{comb}(x) \leq \eta_{aux}(x)$ for all $x \in P$ or $\pi_{comb}(x) \geq \eta_{aux}(x)$ for all $x \in P$.  

If $\pi_{comb}(x) \leq \eta_{aux}(x)$ for all $x \in P$, then $\pi_{sym}(x) = p^{(1)}(x)$ for all $x \in P$ (c.f.~(\ref{eqn:pi_sym})); this holds even if $\pi_{comb}(x) = \eta_{aux}(x)$, as $\pi_{comb}(x) = \eta_{aux}(x)$ implies that $\pi_{sym}(x) = \eta_{aux}(x) = \pi_{comb}(x) \leq \pi_{fill-in}(x)$ (c.f.~Lemma~\ref{lem:pi_fill_in}), so that $\pi_{sym}(x) = \eta_{aux}(x) = p^{(1)}(x)$.  
Otherwise, $\pi_{comb}(x) \geq \eta_{aux}(x)$ for all $x \in P$, and it can be shown similarly that $\pi_{sym}(x) = p^{(2)}(x)$ for all $x \in P$.  In either case, $\pi_{sym}(x)$ is affine on $P$, since $p^{(1)}$ and $p^{(2)}$ are both compatible with $\Pcal$.  
Furthermore, since $\nabla p^{(1)}(x), \nabla p^{(2)}(x), \nabla \eta_{aux}(x) \in \{ \frac{5\epsilon}{12}g_{\delta_3}^{(1)}, \dots, \frac{5\epsilon}{12} g_{\delta_3}^{(n+1)}\}$, whenever the gradients exist, $\pi_{sym}$ is a piecewise linear $(n+1)$-slope function.

(ii)  To demonstrate that $\pi_{sym}$ is  minimal, we first show that this function satisfies (C1).  
We have via Proposition~\ref{prop:pi_comb}, Lemma~\ref{lem:pi_fill_in}, and Lemma~\ref{lem:eta} that $0 = \pi_{comb}(z) = \pi_{fill-in}(z) < \frac{1}{4} \leq \eta_{aux}(z)$ for all $z \in \mathbb Z^n$.  Hence, for all such $z \in \mathbb Z^n$, we have $\pi_{sym}(z)  = 0$, by~(\ref{eqn:pi_sym}).
Additionally, 
\begin{itemize}
\item[(1)] When $\pi_{comb}(x) < \eta_{aux}(x)$, we have
$\pi_{sym}(x) = \min(\pi_{fill-in}(x),\eta_{aux}(x)) \geq 0$,
by~(\ref{eqn:pi_sym}), Lemma~\ref{lem:pi_fill_in}, and Lemma~\ref{lem:eta}.
\item[(2)] Suppose $\pi_{comb}(x) > \eta_{aux}(x)$.  We claim that $x \notin \Lambda_{\delta_3,z}$ for any $z \in \mathbb Z^n$. If, to the contrary, $x \in \Lambda_{\delta_3,z}$ for some $z \in \mathbb Z^n$, then we arrive at the contradiction $\pi_{comb}(x) \leq \frac{5\epsilon}{12} < \frac{1}{4} \leq \eta_{aux}(x)$, where the three inequalities follow from (iii) in Lemma~\ref{lem:pi_fill_in}, the fact that $\epsilon \leq \frac12$, and Lemma~\ref{lem:eta} respectively. 
Hence, our claim holds and 
\begin{align*}
\pi_{sym}(x) & =  1 - \min(\pi_{fill-in}(b-x),\eta_{aux}(b-x))  \geq  1 - \pi_{fill-in}(b-x)  \\
& \geq  1 - \pi_{comb}(b-x) - \frac{\epsilon}{36}  =  \pi_{comb}(x) - \frac{\epsilon}{36} \\
& =  \Big( 1 - \frac{5\epsilon}{6} \Big)\wt\pi(x) + \frac{5\epsilon}{6}\pi_{\delta_3}(x) -  \frac{\epsilon}{36}  \geq  \frac{5\epsilon}{12} - \frac{\epsilon}{36} \geq 0,
\end{align*}
where the first equality follows from~(\ref{eqn:pi_sym}), the second inequality is a consequence of (ii) of Lemma~\ref{lem:pi_fill_in}, the second equality uses the symmetry of $\pi_{comb}$ established in Proposition~\ref{prop:pi_comb}, the third equality follows from~(\ref{eqn:pi_comb}), and the second to last inequality follows from the fact that $\tilde \pi \geq 0$ (c.f.~Proposition~\ref{prop:pwl}) and the fact that $\pi_{\delta_3}(x) \geq \frac12$ by~(\ref{eqn:pi_delta}) as $x \notin \Lambda_{\delta_3,z}$.
%
\item[(3)] When $\pi_{comb}(x) = \eta_{aux}(x)$, we have 
$\pi_{sym}(x)  = \eta_{aux}(x)  \geq 0$,
again by~(\ref{eqn:pi_sym}) and Lemma~\ref{lem:eta}.
\end{itemize}
Hence, $\pi_{sym}$ satisfies (C1).  

We next show that $\pi_{sym}$ satisfies (C2). If $\pi_{comb}(x) < \eta_{aux}(x)$, then  
\[1-\pi_{comb}(b-x) < 1 - \eta_{aux}(b-x) \iff \pi_{comb}(b-x) > \eta_{aux}(b-x),
\]
by Proposition~\ref{prop:pi_comb} and Lemma~\ref{lem:eta}. 
Hence, by~(\ref{eqn:pi_sym}),
\[
\pi_{sym}(x) + \pi_{sym}(b-x) = \min(\pi_{fill-in}(x) ,\eta_{aux}(x)) + 1 - \min(\pi_{fill-in}(x) ,\eta_{aux}(x)) = 1.
\]
A similar argument demonstrates that $\pi_{sym}(x) + \pi_{sym}(b-x) = 1$ when $\pi_{comb}(x) > \eta_{aux}(x)$.  Finally, if $\pi_{comb}(x) = \eta_{aux}(x)$, then by~(\ref{eqn:pi_sym}) and Lemma~\ref{lem:eta},
\[
\pi_{sym}(x) + \pi_{sym}(b-x) = \eta_{aux}(x) + \eta_{aux}(b-x) = 1.
\]
Thus $\pi_{sym}$ satisfies (C2).

We finally show that $\pi_{sym}$ is subadditive.  Consider the following cases:

{\bf Case 1:}  Suppose that $x,y \notin \Lambda_{\delta_3,z}$, $x+y \notin b- \Lambda_{\delta_3,z}$ for all $z \in \mathbb Z^n$.  Note that by Lemma~\ref{lem:pi_fill_in},
\[
\| \pi_{sym} - \pi_{comb} \|_\infty \leq \| \pi_{sym} - \pi_{fill-in} \|_\infty + \| \pi_{fill-in} - \pi_{comb} \|_\infty \leq \frac{\epsilon}{18} + \frac{\epsilon}{36} = \frac{\epsilon}{12},  
\]
as we have already established that $\| \pi_{sym} - \pi_{fill-in} \|_\infty\leq \frac{\epsilon}{18}$ in the first part of this proof.  
Hence, by (iii) of Proposition~\ref{prop:pi_comb},
\begin{align*}
\Delta \pi_{sym}(x,y) &= \pi_{sym}(x) + \pi_{sym}(y) - \pi_{sym}(x+y)\\
					    &\geq  \pi_{comb}(x) + \pi_{comb}(y) - \pi_{comb}(x+y) - 3\cdot\frac{\epsilon}{12} \geq 0.
\end{align*}

{\bf Case 2:}  Suppose that $x$ or $y$ belongs to $\Lambda_{\delta_3,z}$ for some $z \in \mathbb Z^n$.  Without loss of generality take $x \in \Lambda_{\delta_3,z}$.  Then
by Lemma~\ref{lem:pi_fill_in}, we have that $\pi_{fill-in}(x) = \pi_{comb}(x)$.  
Furthermore, we have already shown that $\pi_{comb}(x) \leq \frac{5\epsilon}{12} < \frac{1}{4} \leq \eta_{aux}(x)$ (see part (2) under the proof of (ii), above).  Thus, $\pi_{sym}(x) = \pi_{comb}(x)$ by~(\ref{eqn:pi_sym}).
Now, since $\pi_{sym}$ is a (continuous) piecewise linear function by (i), the gradient of $\pi_{sym}$ exists almost everywhere.  
Additionally, we established above in the proof of (i) that when this gradient exists, it will belong to the set $ \{ \frac{5\epsilon}{12}g_{\delta_3}^{(1)}, \dots, \frac{5\epsilon}{12} g_{\delta_3}^{(n+1)}\}$.  
Thus, we may write 
\begin{align*}
\pi_{sym}(y) - \pi_{sym}(x+y) & = \pi_{sym}(y) - \pi_{sym}(x-z+y) = - \int_0^1 \langle \nabla \pi_{sym}(y+t(x-z)), x-z \rangle\,dt \\
							    & \geq - \max_{i=1}^{n+1} \frac{5\epsilon}{12} \langle g_{\delta_3}^{(i)}, x-z \rangle = - \frac{5\epsilon}{12} \gamma_{\delta_3}(x-z) = -\pi_{comb}(x),
\end{align*}
where the last equality follows from (iii) in Lemma~\ref{lem:pi_fill_in}.
We therefore deduce that
\[
\Delta\pi_{sym}(x,y) = \pi_{sym}(x) + \pi_{sym}(y) - \pi_{sym}(x+y) \geq \pi_{comb}(x) - \pi_{comb}(x) \geq 0.
\]

{\bf Case 3:}  Suppose that $x+y \in b - \Lambda_{\delta_3,z}$ for some $z \in \mathbb Z^n$.  We have
\begin{align*}
\Delta\pi_{sym}(x,y)& = \pi_{sym}(x) + \pi_{sym}(y) - \pi_{sym}(x+y)\\
				       &  = 1-\pi_{sym}(b-x) + \pi_{sym}(y) + \pi_{sym}(b-x-y) -1\\
				       & = \pi_{sym}(b-x-y) +  \pi_{sym}(y) - \pi_{sym}(b-x),
\end{align*}
where we have used the symmetry of $\pi_{sym}$, i.e., $\pi_{sym}$ satisfies (C2), as established above.
Since $b-x-y \in \Lambda_{\delta_3,z}$, an argument similar to that of Case 2 shows us that 
$$\begin{array}{rcl}\pi_{sym}(b-x-y) +  \pi_{sym}(y) - \pi_{sym}(b-x) &= &\pi_{comb}(b-x-y) +  \pi_{sym}(y) - \pi_{sym}((b-x-y) + y) \\
& \geq & \pi_{comb}(b-x-y) -  \pi_{comb}(b-x-y) \geq 0.\end{array}$$

Hence, $\pi_{sym}$ is subadditive, and thus  minimal.

(iii)
It remains to show that $\pi_{sym}$ is a genuinely $n$-dimensional function.  Suppose that $\pi_{sym}$ is not a genuinely $n$-dimensional function, so that there exists a function $\varphi: \mathbb R^{n-1} \rightarrow \mathbb R$ and a matrix $A \in \mathbb R^{(n-1)\times n}$ such that $\pi_{sym}(x) = \varphi(Ax)$ for all $x \in \mathbb R^n$.  Because $A \in \mathbb R^{(n-1)\times n}$, there must exist nonzero $w \in \mathbb R^n$ such that $Aw = 0$.  Furthermore, since $0 \in \text{int}( \Lambda_{\delta_3,0})$, there exists $c > 0$ such that $cw \in \Lambda_{\delta_3,0}$.  
Now, we observe that
\[
 \pi_{fill-in}(cw) = \pi_{comb}(cw) \leq \frac{5\epsilon}{12} < \frac{1}{4} \leq \eta_{aux}(cw),
\]
where the equality follows from (iii) in Lemma~\ref{lem:pi_fill_in} and the inequalities follow from a previous argument (see part (2) under the proof of (ii), above). 
Thus, by~(\ref{eqn:pi_sym})
\[
\pi_{sym}(cw) =  \pi_{fill-in}(cw) = \pi_{comb}(cw) = \frac{5\epsilon}{12} \gamma_{\delta_3}(cw)> 0,
\]
where the last equality follows from (iii) in Lemma~\ref{lem:pi_fill_in}. However, we also can see that
$
\pi_{sym}(cw) = \varphi(Acw) = \varphi(0) = \pi_{sym}(0) = 0
$,
a contradiction.  Hence, $\pi_{sym}$ must be a genuinely $n$-dimensional function.

We conclude that by~\cite[Theorem 1.7]{basu-hildebrand-koeppe-molinaro:k+1-slope}, $\pi_{sym}$ is an extreme function of the set of  minimal functions.
\end{proof}
%
%
\section{Proof of Theorem~\ref{thm:main_thm}}\label{sec:proof}
%
%
We now give a proof of the main result of this paper, Theorem~\ref{thm:main_thm}, by applying the results from the previous sections.  
%
%

Fix $b \in \mathbb Q^n \setminus \mathbb Z^n$, and let $\epsilon > 0$ be given.  
Let $\pi:\mathbb R^n \rightarrow \mathbb R$ be any continuous,  minimal function.  
Without loss of generality, we may take $b \in ([0,1)^n \cap \mathbb Q^n)\setminus\{0\}$, and $\epsilon \in (0,\frac{1}{2}]$; if $\epsilon > \frac{1}{2}$, we may replace $\epsilon$ with $\frac{1}{2}$, establish the result with this replacement, and see that the result will still then hold for the original $\epsilon$.  
Finally, assume that $n  > 1$; if $n =1$, it has already been established that this result holds by~\cite[Theorem 2]{basu2016minimal}.

There exists $\wt \pi: \mathbb R^n \rightarrow \mathbb R$ satisfying the conclusions of Proposition~\ref{prop:pwl} for some suitable $\wt \delta > 0$ (c.f.~(\ref{eqn:delta_12_1}),~(\ref{eqn:delta_12_2}), and the proof of Proposition~\ref{prop:pwl}), in particular, $\| \pi - \wt \pi\|_\infty < \frac{\epsilon}{36}$.  
Choose an appropriate $\delta_3 > 0$ (c.f.~(\ref{eqn:delta_3})), and form $\pi_{comb}: \mathbb R^n \rightarrow \mathbb R$ via~(\ref{eqn:pi_comb}).  We have that $\pi_{comb}$ satisfies the conclusions of Proposition~\ref{prop:pi_comb}, including $\| \wt \pi - \pi_{comb}\|_\infty \leq \frac{5\epsilon}{6}$.  
Next choose $\delta_4 > 0$ that satisfies~(\ref{eqn:delta_4}), and construct $\pi_{fill-in}: \mathbb R^n \rightarrow \mathbb R$ via~(\ref{eqn:pi_fill_in}).  By Lemma~\ref{lem:pi_fill_in}, we have $\|\pi_{comb} - \pi_{fill-in} \|_\infty \leq \frac{\epsilon}{36}$.  
Furthermore, if we assemble $\pi_{sym}: \mathbb R^n \rightarrow \mathbb R$ as in~(\ref{eqn:pi_sym}), we have by Proposition~\ref{prop:pi_sym} that $\|\pi_{fill-in} - \pi_{sym} \|_\infty \leq \frac{\epsilon}{18}$ and that $\pi_{sym}$ is an extreme function of the set of  minimal functions.
Now,
\begin{align*}
\| \pi - \pi_{sym}\|_\infty &\leq \| \pi - \wt\pi\|_\infty + \| \wt\pi - \pi_{comb}\|_\infty + \| \pi_{comb} - \pi_{fill-in}\|_\infty + \| \pi_{fill-in} - \pi_{sym}\|_\infty \\
&< \frac{\epsilon}{36} +  \frac{5\epsilon}{6} +  \frac{\epsilon}{36} +  \frac{\epsilon}{18}  < \epsilon.
\end{align*}
We complete the proof by setting $\pi^\prime \equiv \pi_{sym}$.

\bibliographystyle{plain}
\bibliography{full-bib}

\end{document}